\documentclass[a4paper,10pt]{article}
\usepackage{fullpage}
\usepackage{amsmath,amssymb,latexsym,amsthm}
\usepackage[pdftex]{graphicx}
\usepackage{amsmath,amssymb}
\usepackage{graphicx}
\usepackage{color, epsfig}
\usepackage{subfigure}
\usepackage{relsize}
\providecommand{\R}{\mathbb{R}}

\providecommand{\Z}{\mathbb{Z}}
\newtheorem{theorem}{Theorem}
\newtheorem{definition}{Definition}

\newtheorem{proposition}{Proposition}
\newtheorem{lemma}{Lemma}
\newtheorem{remark}{Remark}
\title{On 
compactness estimates\\ for hyperbolic systems of conservation laws}
\author{
Fabio Ancona\footnote{
Dipartimento di Matematica,
Universit\`a  di Padova,
Via Trieste 63, 35121 Padova, Italy.
E-mail: \texttt{ancona@math.unipd.it}
} ,
Olivier Glass\footnote{
Ceremade,
Universit\'e Paris-Dauphine, UMR CNRS 7534,
Place du Mar\'echal de Lattre de Tassigny, 
75775 Paris Cedex 16, France.
E-mail: \texttt{glass@ceremade.dauphine.fr}
} ,
Khai T. Nguyen\footnote{
Department of Mathematics,
Penn State University,
235, McAllister buiding, PA 16802, USA.
E-mail: \texttt{ktn2@math.psu.edu}
}
\\
\noalign{\bigskip\medskip}
}
\begin{document}
\maketitle
\begin{abstract}
We study the compactness in $L^{1}_{loc}$ of the semigroup mapping $(S_t)_{t > 0}$ 
defining entropy weak solutions
of general hyperbolic systems of conservation laws in one space dimension. 
We establish a lower estimate
for the Kolmogorov $\varepsilon$-entropy of the image through the mapping $S_t$ of bounded sets in $L^{1}\cap L^\infty$,
which is
of the same order $1/\varepsilon$ as the ones established by the authors
for scalar conservation laws. We also provide an upper estimate of  order $1/\varepsilon$ for the Kolmogorov $\varepsilon$-entropy
of such sets in the case of Temple systems with genuinely nonlinear characteristic families, that extends
the same type of estimate derived by De Lellis and Golse for scalar conservation laws with convex flux.
As suggested
by Lax, these quantitative compactness estimates  could 
provide a measure of the order of  ``resolution'' of the numerical methods implemented for these equations.
\end{abstract}
%
%
%
%
%
%
%
%
%
%
%
%
%
\section{Introduction}
Consider a general system of hyperbolic conservation laws in one space dimension
\begin{equation}
\label{conlaws}
u_t+f(u)_x=0,
\qquad\ t\geq 0,\, x\in\R\,,
\end{equation}
where $u=u(t,x)\in \mathbb{R}^N$ represents the conserved quantities and  the flux $f(u)=(f_1(u),\dots,f_N(u))$
is a vector valued map of class $C^2$, defined on an open, connected domain $\Omega\subseteq\R^N$
containing the origin.
 Assume that the above system is strictly hyperbolic, i.e, that the Jacobian matrix $Df(u)$ has $N$ real, distinct 
 eigenvalues $\lambda_1(u)<...<\lambda_N(u)$ for all $u\in\Omega$. 
Several laws of physics take the form of a conservation equation. 
A primary example of such systems is provided by the Euler equations
of non-viscous gases (cf.~\cite{Dafermos:Book}).
The fundamental paper of Bianchini and Bressan~\cite{BB}
shows that~\eqref{conlaws} generates a unique (up to the domain) Lipschitz continuous
semigroup $S : [0,\infty[ \times \mathcal{D}_0\to \mathcal{D}_0$
defined on a closed domain $\mathcal{D}_0\subset L^1(\R, \R^N)$,
with the properties:
\begin{enumerate}
  \item[(i)] 
  \begin{equation}
  \label{domainsgr}
  \Big\{v\in L^1(\R, \Omega)\,\big|\, \text{Tot.Var.}(v)\leq \delta_0\Big\}\subset\mathcal{D}_0\subset
  \Big\{v\in L^1(\R, \Omega)\,\big|\, \text{Tot.Var.}(v)\leq 2\delta_0\Big\},
  \end{equation}
  for suitable constant $\delta_0>0$. 
  \item[(ii)] For every $\overline u\in \mathcal{D}_0$, the semigroup trajectory $t\mapsto S_t \overline u\doteq u(t,\cdot)$
  provides an entropy weak solution of the Cauchy problem for~\eqref{conlaws}, with
  initial data 
\begin{equation}
\label{indata}
u(0,\cdot)=\overline u,
\end{equation}
that satisfy the following admissibility criterion proposed by T.P. Liu in \cite{tplrpnpn},
which generalizes the classical stability conditions introduced by Lax~\cite{lax}.\newline
%
\textbf{Liu  stability condition.} A shock discontinuity of the $i$-th family $(u^L,\,u^R)$, traveling
with speed $\sigma_i[u^L,\,u^R]$, is {\it Liu admissible} if, for any
state $u$ lying on the $i$-th Hugoniot curve 
between $u^L$ and
$u^R$, the shock speed $\sigma_i[u^L,u]$ of the discontinuity
$(u^L,u)$ satisfies
\begin{equation}
  \label{liuc}
  \sigma_i[u^L,\,u]\geq\sigma_i[u^L,\,u^R]\,.
\end{equation}
%
%
\end{enumerate}
Thanks to the uniform BV-bound on the elements of $\mathcal{D}_0$, applying
Helly's compactness theorem it follows that $S_t$  is a compact mapping,
for every $t> 0$. Aim of this paper is to provide a quantitative estimate 
of the compactness of such a mapping. Namely, following a suggestion of Lax~\cite{Lax54},
we wish to estimate the Kolmogorov $\varepsilon$-entropy in $L^1$ of the image
of bounded sets in $\mathcal{D}_0$ through the  map $S_t$. We recall that, 
given a metric space $(X,d)$, and  a totally bounded subset $K$ of $X$, 
for every $\varepsilon>0$ we define the Kolmogorov  $\varepsilon$-entropy of $K$
as follows. Let $N_{\varepsilon}(K \ | \ X)$ 
be the minimal number of sets in a cover of $K$ by subsets of~$X$ having diameter no larger than $2\varepsilon$. Then, the 
{\it $\varepsilon$-entropy} of $K$ is defined as
\begin{equation} \nonumber
H_{\varepsilon}(K \ | \ X) \doteq \log_{2} N_{\varepsilon}(K \ | \ X).
\end{equation}
Throughout the paper, we will call an {\it $\varepsilon$-cover}, a cover of $K$ by subsets of $X$ having diameter no larger than $2\varepsilon$. Entropy numbers play a central roles in various areas of information theory and statistics
as well as of learning theory. In the present setting, this concept  could 
provide a measure of the order of ``resolution'' of a numerical method for~\eqref{conlaws}, as suggested in~\cite{Lax78}.
\par

In the case of scalar conservation laws ($N=1$) with strictly convex (or concave) flux, De Lellis and Golse~\cite{DLG}
obtained an upper bound  of order $1/\varepsilon$ on the $\varepsilon$-entropy of $S_t (\mathcal{L})$, for sets $\mathcal{L}\subset L^1(\R)$  of bounded,
compactly supported functions, 
of the form
\begin{equation}
\label{Cclass}
\mathcal{L}_{[I,m,M]}\ \doteq\ \Big\lbrace{\overline u\in L^1(\mathbb{R},\Omega)\,\big|\ \mathrm{Supp}(\overline u)
\subset I,
\ \|\overline u\|_{L^1}\leq m, \|\overline u\|_{L^{\infty}}\leq M \Big\rbrace}\,,
\end{equation}
where $I$ denotes a given interval of $\R$.
This upper bound turns out to be optimal since  we  provided in~\cite{AOK} a
lower bound of the same order  for the $\varepsilon$-entropy of $S_t (\mathcal{L})$,
for sets $\mathcal{L}$ as in~\eqref{Cclass},
thus showing that such an $\varepsilon$-entropy is of size $\approx(1/\varepsilon)$
for scalar conservation laws with strictly convex (or concave) flux.

These estimates hold 
for  classes of initial data
with possibly unbounded total variation
 because the regularizing effect due to the
convexity (or concavity) of the flux function $f$ yields solutions $u(t,\cdot)$ of~\eqref{conlaws} that belong to $\text{BV}_\text{loc}(\R)$ for any $t>0$. This is no more true in the  case of conservation laws with non convex (or concave) flux and in
the case of systems of conservation laws with no monotonicity assumption on the
eigenvalues of the
Jacobian matrix $Df(u)$. On the other hand, the well-posedness theory
for a general system of conservation laws has been established only for initial data with sufficiently small total variation. 
Therefore, aiming to establish estimates on  the $\varepsilon$-entropy of solutions to general systems of conservation laws~\eqref{conlaws}, it is natural to restrict our analysis to classes of initial data 
with uniformly bounded total variation.
Namely, 
we shall provide estimates on the
$\varepsilon$-entropy of $S_t (\mathcal{L}\cap \mathcal{D}_0)$,
for sets~$\mathcal{L}$ as in~\eqref{Cclass}, with ~$\mathcal{D}_0$ as in~\eqref{domainsgr}.
Specifically, 
we prove the following.
\begin{theorem}
\label{mainthm}
Let $f:\Omega \to \R^N$ be a $C^2$ map on an open, connected domain $\Omega\subset\R^N$ containing 
the origin,
and assume that the system~\eqref{conlaws} is strictly hyperbolic.
%
%
Let $(S_t)_{\geq 0}$ be the semigroup of entropy weak solutions generated by~\eqref{conlaws} 
defined on a domain $\mathcal{D}_0$   satisfying~\eqref{domainsgr}.
Then, given any $L,m,M, T>0$, for any interval $I\subset\R$ of length $|I|= 2L$,
and for $\varepsilon>0$ sufficiently small,
the following estimates hold.
\begin{enumerate}  
\item[(i)] 
\begin{equation}
\label{entrlowerbound}
H_{\varepsilon}\Big(S_T\big(\mathcal{L}_{[I,m,M]}\cap{\mathcal{D}_0}\big)\ |\ L^1(\mathbb{R},\Omega)\Big)\geq 
\frac{N^2 L^2}{T}\cdot \frac{\big(\min\big\{c_1, c_2 \frac{T}{L}\big\}\big)^{\!2}}{\max\big\{c_3,\, c_4 \frac{N^2 L}{T},\,  c_5 \frac{N L}{\delta_0 T}\big\}}
\cdot\frac{1}{\varepsilon}\,,
\end{equation} 
where 
$c_3\geq 0$,
$c_l>0$, $l=1, 2, 4, 5$, are  constants given in~\eqref{c1234def}, \eqref{c-def},
which depend only on 
the eigenvalues $\lambda_i(u)$ of the
Jacobian matrix $Df(u)$, on the corresponding right and left eigenvectors $r_i(u), l_i(u)$,
and on their derivatives, in a neighbourhood of the origin.
%
\medskip
\item[(ii)] 
\begin{equation}
\label{entrupperbound}
H_{\varepsilon}\Big(S_T\big(\mathcal{L}_{[I,m,M]}\cap{\mathcal{D}_0}\big)\ |\ L^1(\mathbb{R},\Omega)\Big)\leq 
48 N\delta_0 \cdot L_T\cdot\dfrac{1}{\varepsilon}\,,
\end{equation}
where 
\begin{equation}
\label{LT}
L_T\doteq L+\frac{\Delta_{\vee}\lambda}{2}\cdot T,
\qquad\quad
\Delta_\vee\lambda\doteq\sup\big\{\lambda_N(u)-\lambda_1(v)\,;\ u,v\in \Omega\big\}.
\end{equation}
\end{enumerate}
\end{theorem}
\begin{remark} 
\label{rem-compar-td}
If the bound $\delta_0$ on the total variation of the initial data in the domain~$\mathcal{D}_0$
satisfies the inequality $\delta_0<\min\big\{
\frac{c_5}{c_3}\!\cdot\! \frac{NL}{T},\, \frac{c_5}{c_4}\!\cdot\!\frac{1}{L}\big\}$
$($interpreting ${1}/{c_3}\doteq \infty$ when $c_3=0$$)$, 
then the lower estimate~\eqref{entrlowerbound}
takes the form
\begin{equation}
\label{entrlowerbound2}
H_{\varepsilon}\Big(S_T\big(\mathcal{L}_{[I,m,M]}\cap{\mathcal{D}_0}\big)\ |\ L^1(\mathbb{R},\Omega)\Big)\geq 
N L \delta_0\cdot\frac{\big(\min\big\{c_1, c_2 \frac{T}{L}\big\}\big)^{\!2}}{c_5} 
\cdot\frac{1}{\varepsilon}\,.
\end{equation} 
Therefore, in this case, upper and lower bounds \eqref{entrupperbound}, \eqref{entrlowerbound2} 
of the $\varepsilon$-entropy turn out to have the same size $N L \delta_0\cdot \frac{1}{\varepsilon}$.
On the other hand, if $c_3>0$, in the case where $T\geq\max \big\{\frac{c_1}{c_2}\!\cdot\! L,\, \frac{c_4}{c_3}\!\cdot\! N^2L,\,
\frac{c_5}{c_3}\!\cdot\! \frac{NL}{\delta_0}\big\}$,
we obtain by~\eqref{entrlowerbound},~\eqref{c1234def},
the estimate
\begin{equation}
\label{entrlowerbound3}
H_{\varepsilon}\Big(S_T\big(\mathcal{L}_{[I,m,M]}\cap{\mathcal{D}_0}\big)\ |\ L^1(\mathbb{R},\Omega)\Big)\geq 
\frac{N^2 L^2}{T}\cdot\frac{c_1^2}{c_3}\cdot\frac{1}{\varepsilon}\,,
\end{equation} 
with $c_3\doteq 2\sup\big\{|\nabla\lambda_i(u)|\,; \ |u|\leq \overline d,\ i=1,\dots,N\,\big\}$
for some $\overline d>0$.
Hence, if $c_3>0$,
for times $T$ sufficiently large we obtain 
a lower bound on the $\varepsilon$-entropy of $S_T(\mathcal{L}_{[I,m,M]}\cap{\mathcal{D}_0})$
which is of the same order ${L^2}/{(|f''(0)|\, T)}\cdot\frac{1}{\varepsilon}$
established in~\cite{AOK} for solutions to scalar
conservation laws with strictly convex (or concave) flux $f$.
\end{remark}
\begin{remark} 
When $N=1$, the semigroup map $S_t$ is defined on the whole space $L^1(\R)$.
Thus, in this case we may analyze the $\varepsilon$-entropy 
of $S_t (\mathcal{L})$
for sets $\mathcal{L}$ of initial data with possibly unbounded total variation
as in~\eqref{Cclass}.
In fact, for scalar conservation laws, with
the same arguments used to establish Theorem~\ref{mainthm}-(i),
if $\overline c\doteq \sup\big\{|f''(u)|\,; \ |u|\leq \overline d\big\}>0$
for some $\overline d>0$,
 one can derive, for $\varepsilon$ sufficiently small, the lower bound
 (cf. Remark~\ref{rem-scalar-control} and Remark~\ref{rem-scalar-combinatory}):
\begin{equation}
\label{entrlowerbound4}
H_{\varepsilon}\Big(S_T\big(\mathcal{L}_{[I,m,M]}\big)\ |\ L^1(\mathbb{R})\Big)
\geq 
\frac{L^2}
{144\cdot\ln(2)\cdot
\overline c \cdot T}\cdot\frac{1}{\varepsilon}.
\end{equation} 
Thus, Theorem~\ref{mainthm} provides in particular an extension of~\cite[Theorem~1.3]{AOK}
to the case of general scalar conservation laws
with smooth, not necessarily convex (or concave) flux.
Clearly, the lower bound~\eqref{entrlowerbound4} is significative only in the case
where $\inf\big\{ |u|\,; \ |f''(u)|>0\big\}=0$, since otherwise one can easily see that the left-hand
side of~\eqref{entrlowerbound4} equals $+\infty$ for small $\varepsilon$.
\end{remark}

The upper bound~\eqref{entrupperbound} stated in Theorem~\ref{mainthm} can be easily obtained relying on the upper estimates
for the covering number of classes of functions with uniformly bounded total variation established in~\cite{BKP}.
In fact, given any element $\varphi\in S_T(\mathcal{L}_{[I,m,M]}\cap{\mathcal{D}_0})$, 
with $\mathcal{L}_{[I,m,M]}$ as in~\eqref{Cclass}, $|I|=2L$,
by the finite speed of propagation along (generalized) characteristics
(cf.\cite[Chapter~10]{Dafermos:Book}) we have the bound $|\mathrm{Supp}(\varphi)|\leq 2 L_T$ on the support of 
$\varphi$,
with $L_T$ as in~\eqref{LT}.
Moreover,  observe that, defining the total variation of a vector valued map $\varphi=(\varphi_1,\dots,\varphi_p):\R\to\R^p$
as $ \text{Tot.Var.}(\varphi)\doteq \sum_i  \text{Tot.Var.}(\varphi_i)$,
and setting
\begin{equation}
\mathcal{M}_{[L,\delta_0, p]}\doteq 
\Big\{\varphi\in\text{BV}\big([0,2L], \R^p\big)\ \big|\ \text{Tot.Var.}(\varphi)\leq 2\delta_0
\Big\},
\end{equation}
one has 
\begin{equation}
\label{entr-bv}
N_{\varepsilon}\Big(\mathcal{M}_{[L,\delta_0, N]}\ \big|\ L^1\big([0,2L], \R^N\big)\Big)\leq 
N_{\varepsilon}\Big(\mathcal{M}_{[NL,\delta_0, 1]}\ \big|\ L^1\big([0,2NL], \R\big)\Big).
\end{equation}
This is due to the fact that, if we let $\varphi\!\!\restriction_{J}$ denote the restriction of a map $\varphi$
to a set $J$, for every $\varepsilon$-cover $\cup_\alpha E^\alpha$ of $\mathcal{M}_{[NL,\delta_0, 1]}$
we can always consider the sets $E^\alpha_1\times\dots\times E^\alpha_N$,
with $E^\alpha_i\doteq\{\varphi(\cdot-(i-1)L)\!\!\restriction_{[(i-1)L, iL]}\,;\varphi\in E^\alpha\}$, 
which provide an $\varepsilon$-cover $\cup_\alpha (E^\alpha_1 \times\dots\times \mathcal{E}^\alpha_N)$
of $\mathcal{M}_{[L,\delta_0, N]}$, with the same cardinality as $\cup_\alpha E^\alpha$.
Thus, given any $L, m, M, T>0$
and  any interval $I\subset\R$ of length $|I|= 2L$,
 applying~\cite[Theorem~1]{BKP},
 and relying on~\eqref{domainsgr}, \eqref{entr-bv}, 
 for $\varepsilon>0$ sufficiently small we find the following upper bound on the minimal covering number
\begin{equation}
\label{entrupperbound2}
\begin{aligned}
N_{\varepsilon}\Big(S_T\big(\mathcal{L}_{[I,m,M]}\cap{\mathcal{D}_0}\big)\ |\ L^1(\mathbb{R})\Big)&\leq 
N_{\varepsilon}\Big(\mathcal{M}_{[L_T,\delta_0, N]}\ \big|\ L^1\big([0,2L_T], \R^N\big)\Big)
\\
&\leq N_{\varepsilon}\Big(\mathcal{M}_{[N L_T,\delta_0, N]}\ \big|\ L^1\big([0,2 N L_T], \R\big)\Big)
\\
&\leq
2^{\strut\frac{48 \delta_0 \cdot N L_T}{\varepsilon}}.
\end{aligned}
\end{equation}
One then clearly recovers~\eqref{entrupperbound} from~\eqref{entrupperbound2}.

Therefore, the main novelty of the estimates stated in Theorem~\ref{mainthm} consists in the
lower bound~\eqref{entrlowerbound} that is independent on the total variation
of the functions in $\mathcal{D}_0$, for times $T$ sufficiently large (cf.~Remark~\ref{rem-compar-td}).
Following the same strategy adopted in~\cite{AOK} we shall prove~\eqref{entrlowerbound} in two steps:
%
\begin{itemize}
\item[1.] For every $i$-th characteristic family, let $s \mapsto R_i(s)$ denote the integral curve
of the $i$-th eigenvector~$r_i$, starting at the origin. Consider a family of profiles of $i$-simple waves $\{\phi_i^\iota\}_{\iota}$
defined as parametrizations $s \mapsto\phi_i^\iota(s)\doteq R_i(\beta_\iota(s))$ of $R_i$ through
a suitable class of 
piecewise affine, compactly supported functions~$\{\beta_\iota\}_\iota$.
We will show that, at any given time $T$, 
 any superposition $\phi^{{\iota_1}, \dots, {\iota_N}}$ of simple waves $\phi_1^{\iota_1}, \dots, \phi_N^{\iota_N}$,
can be obtained as the value $u(T,\cdot)=S_T \overline u$ of an entropy admissible weak solution of~\eqref{conlaws},
with initial data $\overline u\in{\mathcal L}_{[L,m,M]}\cap{\mathcal{D}_0}$.

\item[2.] 
We shall provide an optimal estimate of the maximum number 
of elements of the family\linebreak $\{\phi^{{\iota_1}, \dots, {\iota_N}}\}_{{\iota_1}, \dots, {\iota_N}}$
contained in a subset of $S_T({\mathcal L}_{[I,m,M]}\cap{\mathcal{D}_0})$ of diameter~$2 \varepsilon$.
This estimate is established
with a similar combinatorial argument as the one used in~\cite{BKP},
and immediately yields a lower bound on the $\varepsilon$-entropy of the set $\{\phi^{{\iota_1}, \dots, {\iota_N}}\}_{{\iota_1}, \dots, {\iota_N}}$.
In turn, from the lower bounds on $H_\varepsilon(\{\phi^{{\iota_1}, \dots, {\iota_N}}\}_{{\iota_1}, \dots, {\iota_N}}\ |\ L^1(\R, \Omega))$,
we recover~\eqref{entrlowerbound}.
\end{itemize}
%

Next we focus our attention on a particular
class of hyperbolic systems introduced by Temple~\cite{serre, temple},
under the assumption that all characteristic families are genuinely nonlinear 
or linearly degenerate (see Definition~\ref{def-temple}
in Subsection~\ref{subsec-control}).
Systems of this type arise in traffic flow models, in multicomponent chromatography,
as well as in problems of oil reservoir simulation.
The special geometric features of such systems allow the existence of a 
continuous semigroup of 
solutions 
$S: [0,\infty[ \times \mathcal{D}\to \mathcal{D}$
defined on domains $\mathcal{D}$
 of $L^\infty$-functions with possibly unbounded variation of the form
 \begin{equation}
 \label{domtemple}
 \mathcal{D}\doteq
  \Big\{v\in L^1(\R, \Omega)\ \big|\ W(v(x))\in [a_1, b_1]\times\cdots\times [a_n,b_n] \ \ \text{for all} \ x\in\R\Big\},
 \end{equation}
 where $W(v)=(W_1(v), \dots, W_N(v))$ denotes the Riemann coordinates
 of $v\in\Omega$  (see~\cite{brgoa1}, \cite{bianc}).

 Every trajectory of the semigroup $t\mapsto S_t \overline u\doteq u(t,\cdot)$ yields
 an entropy weak solution of~\eqref{conlaws},\eqref{indata}.
 When all characteristic families are genuinely nonlinear
 such a semigroup is Lipschitz continuous and the map $u(t,x)\doteq S_t \overline u(x)$
satisfies the following
Ole\v{\i}nik-type inequalities on the decay of positive waves (expressed
in Riemann coordinates $w_i(t,\cdot)\doteq W_i(u(t,\cdot))$):
\begin{equation}
\label{Olest}
\frac{w_i(t,y)-w_i(t,x)}{y-x}\leq\frac{1}{c\,t}
\qquad \forall~x<y,\quad t>0,\quad i=1,\dots,N,
\end{equation}
for some constant 
\begin{equation}
\label{c-def-temple}
0<c\leq\inf\Big\{|\nabla \lambda_i(u)\cdot r_i(u)|\,;\, u\!\in\! W^{-1}(\Pi), \, i=1,\dots,N\Big\},
\end{equation}
where
\begin{equation*}
\Pi \doteq [a_1, b_1]\times\cdots\times [a_n,b_n].
\end{equation*}
In this setting, it is natural to ask whether we can extend the estimates provided by Theorem~\ref{mainthm}
to classes of initial data with unbounded variation.
The next result provides a positive answer to this question.
Namely, relying on the analysis of the evolution of the Riemann coordinates along
the characteristics and on the Ole\v{\i}nik-type inequalities, we will establish upper and lower estimates on
the $\varepsilon$-entropy of solutions to genuinely nonlinear Temple systems
which are the natural extension to this class of hyperbolic
systems of the compactness estimates
established
in~\cite{AOK, DLG}
 for scalar conservation laws with strictly convex (or concave) flux.
Specifically, letting $S^w_t \overline w\doteq W(u(t,\cdot))$ denote
the Riemann coordinates expression of the
solution of~\eqref{conlaws},\eqref{indata}, with $\overline u\doteq W^{-1}\circ \overline w$, determined by the semigroup map $S$,
and adopting the norms $\|w\|_{L^1}\doteq \sum_i \|w_i\|_{L^1}$,
$\|w\|_{L^\infty}\doteq \sup_i \|w_i\|_{L^\infty}$
on the space $L^1(\R, \Pi)$,
we prove the following
\begin{theorem}
\label{templethm}
In the same setting of Theorem~\ref{mainthm},
assume that ~\eqref{conlaws} is a strictly hyperbolic system of Temple class, 
and that all  characteristic families are genuinely nonlinear or linearly degenerate.
Let $(S_t)_{\geq 0}$ be the semigroup of entropy weak solutions  generated by~\eqref{conlaws} 
defined on a domain $\mathcal{D}$  as in~\eqref{domtemple}.
Then, given 
any $L,m, M, T>0$, and any interval $I\subset\R$ of length $|I|= 2L$,
setting
\begin{equation}
\label{Cclass-w}
\mathcal{L}^w_{[I,m,M]}\ \doteq\ \Big\lbrace{\overline w\in L^1(\mathbb{R},\Pi)\,\big|\ \mathrm{Supp}(\overline w)
\subset I,
\ \|\overline w\|_{L^1}\leq m, \|\overline w\|_{L^{\infty}}\leq M \Big\rbrace}\,,
\end{equation}
for $\varepsilon>0$ sufficiently small,
the following hold.
\begin{enumerate}  
\item[(i)] 

\begin{equation}
\label{entrlowerbound-temple}
H_{\varepsilon}\Big(S^w_T\big(\mathcal{L}^w_{[I,m,M]}\big)\ |\ L^1(\mathbb{R},\Pi)\Big)\geq 
\frac{N^2 L^2}{T}\cdot \frac{1}{\max\big\{c_6,\,  c_7 \frac{N L}{T}\big\}}
\cdot\frac{1}{\varepsilon}.
\end{equation} 
where $c_6, c_7$ are nonnegative constants given in~\eqref{c1234def2}, \eqref{c-def2},
which depend only on the gradient of the eigenvalues $\lambda_i(u)$ of the
Jacobian matrix $Df(u)$ and on the corresponding right eigenvectors~$r_i(u)$, in a neighbourhood of the origin.
\medskip
\item[(ii)] 
If all characteristic families are genuinely nonlinear, one has
\begin{equation}
\label{entrupperbound-temple}
H_{\varepsilon}\Big(S^w_T\big(\mathcal{L}^w_{[I,m,M]}\big)\ |\ L^1(\mathbb{R},\Pi)\Big)\leq 
\frac{32 N^2 L_T^2}{c\,T}\cdot\dfrac{1}{\varepsilon},
\end{equation}
where 
\begin{equation}
\label{lt-def}
L_T\!\doteq L + \sqrt{\frac{8N mT}{c}} \cdot\sup\Big\{
|\nabla \lambda_i(u)\cdot r_j(u)|\,;\,  |W(u)|\leq M, \, i,j=1,\dots,N\Big\},
\end{equation}
and $c$ is the constant appearing in~\eqref{Olest}.
\end{enumerate}
\end{theorem}
\medskip

The paper is organized as follows. In Section~\ref{sec:simple-wave} we first introduce a family of simple waves and then 
construct a class of classical solutions of~\eqref{conlaws} with initial data given by the profiles of $N$ simple waves
supported on disjoint sets. 
This analysis is in particular carried out with a finer accuracy for the special class of Temple systems.
In Section~\ref{sec:lower-est} we establish a controllability result and a combinatorial computation
both for general hyperbolic systems and for Temple systems, which yield the lower bound on the $\varepsilon$-entropy
stated in Theorem~\ref{mainthm} and Theorem~\ref{templethm}.
Finally, Section~\ref{sec:upper-est-temple} contains the derivation of the upper bound on the $\varepsilon$-entropy
for Temple systems stated in Theorem~\ref{templethm}.
\section{Simple waves and classical solutions}
\label{sec:simple-wave}
\subsection{Simple waves}
\label{subsec:Simple waves}

Let $f:\Omega \to \R^N$ be a $C^2$ map on an open, connected domain $\Omega$,
and assume that a neighbourhood of the origin $B_{\overline d}\doteq \{u\in\R^n\, |\, |u|\leq\overline d\,\}$
is contained in $\Omega$.
We shall consider here a class of continuous, piecewise $C^1$ solutions of~\eqref{conlaws} 
that take values on the integral curves of the eigenvectors of the Jacobian matrix $Df$. Such solutions
can be regarded as the nonlinear analogue
of the elementary waves of each characteristic family in which it is
decomposed a solution of a semilinear system (cf.~\cite[Section~7.6]{Dafermos:Book}).
For every $i$-th characteristic
family, let $s\mapsto R_i(s)$ denote the integral curve of the eigenvector $r_i$, passing through the origin.
More precisely, we define $R_i(\cdot)$ as the unique solution of the Cauchy problem
\begin{equation}
\label{j-rarefaction}
\frac{d u}{ds}=r_i(u(s)),\quad u(0)=0,
\end{equation}
%
that we may assume to be defined on the interval  $]-\overline d, \overline d\,[\,$
of the same size of the neighbourhood $B_{\overline d}\subset \Omega$.
The curve $R_i$ is called the $i$-{\it rarefaction curve} through $0$.
We may select the basis of right eigenvectors $r_i(u), i=1,\dots, N$, together with a basis
of left eigenvectors $l_i, i=1,\dots, N$,  so that 
\begin{equation}
\label{normalized-eigenvect}
|r_i|\equiv 1,\qquad\quad
l_i\cdot r_j \equiv\delta_{i,j}\doteq
\begin{cases}
1\ &\text{if}\quad i=j,
\\
0\ &\text{if}\quad i\neq j,
\end{cases}
\end{equation}
where $u\cdot v$ denotes the inner product of the vectors $u,v\in\R^N$.
It follows in particular that 
\begin{equation}
\label{Rparam}
|R_i(s)|\leq |s|\qquad \forall~s\in\,]-\overline d, \overline d\,[\,.
\end{equation}
%
%
For every $b>0$, $0<d< \overline d$,  we define the class of functions 
\begin{equation}
\label{PCset}
\mathcal{PC}^1_{[d,b]}\doteq
\Big\{\beta:\R\to [-d, d]\ \big|\ \beta\ \text{is piecewise} \ C^1\ \text{and}\ |\dot\beta(x)|\leq b
\Big\}.
\end{equation}
Here, we say that a map $\beta:\R\to [-d, d]$ is {\it piecewise} $C^1$ if $\beta$ is continuous on $\R$
and continuously differentiable on all but finitely many points of $\R$, while the bound on $\dot\beta$
in~\eqref{PCset} is assumed to be satisfied at every point of differentiability of $\beta$.
%
Given $\beta\in\mathcal{PC}^1_{[d, b]}$, consider the map
\begin{equation}\label{j-initial wave}
\phi^\beta_i(x)\doteq R_i(\beta(x))
\qquad x\in\R\,,
\end{equation}
and define the corresponding $i$-th characteristic starting at $y\in\R$ as:
\begin{equation}
\label{i-character}
x_i(t,y)\doteq y+\lambda_i(\phi^\beta_i(y))\cdot t,\qquad  t\geq 0.
\end{equation}
Observe that, by\eqref{j-rarefaction}, one has
\begin{equation}
\label{dphi}
\frac{d}{dx} \phi^\beta_i(x)=\beta'(x)\cdot r_i(\phi^\beta_i(x))
\end{equation}
at every point $x$ of differentiability of $\beta$.
Hence, differentiating~\eqref{i-character} w.r.t. $y$ at a point
where $\beta$ is differentiable we find
\begin{equation}
\label{d-i-character}
\frac{\partial}{\partial y}x_i(t,y)=1+[\nabla\lambda_i(\phi^\beta_i(y))\cdot r_i(\phi^\beta_i(y))]\cdot \dot{\beta}(y)\cdot t\,,
\qquad t\geq 0\,.
\end{equation}
Set
\begin{equation}
\label{lambdadef}
\alpha_1\doteq \sup\big\{|\nabla\lambda_i(u)|\,; \ u\in B_{\overline d},\ i=1,\dots,N\,\big\}.
\end{equation}
Then, relying on~\eqref{normalized-eigenvect}, \eqref{Rparam}, \eqref{lambdadef}, 
and because of the bound on $\dot\beta$ in~\eqref{PCset},
we derive from~\eqref{d-i-character} the inequality
\begin{equation}
\frac{\partial}{\partial y}x_i(t,y)\geq 1-\alpha_1 \, b\cdot t\,,
\qquad t\geq 0\,,
\end{equation}
which, in turn, yields
\begin{equation}
\label{d-i-characterbound}
 \frac{\partial}{\partial y}x_i(t,y)\geq \frac{1}{2}
\qquad \forall~ t \in [0, 1/(2\alpha_1\cdot b)].
\end{equation}
%
The 
inequality in~\eqref{d-i-characterbound}, in particular,
implies that the map $y\mapsto x_i(t,y)$  is increasing, hence injective. 
Moreover, since $\phi^\beta_i$ is continuous,
from 
~\eqref{d-i-characterbound}  
we deduce also that the image of $y\mapsto x_i(t,y)$  is the whole line $\R$.
Therefore, for every fixed $0\leq t\leq1/{(2\alpha_1\cdot b)}$,
 we may define the inverse map of $x_i(t,\cdot)$ on $\R$.
Then, set
\begin{equation}
\label{inverse-y}
z_i(t,\cdot)\doteq\ x_i^{-1}(t,\cdot),
\end{equation}
and define the function
\begin{equation}
\label{solution-u}
u(t,x)\doteq \phi^\beta_i(z_i(t,x)),\quad\forall (t,x)\in [0,T]\times\R\,,
\end{equation}
with $T\leq 1/{(2\alpha_1\cdot b)}$.
The next lemma shows that $u(t,x)$ provides a classical
solution
of~\eqref{conlaws} on $[0,T]\times\R$, and we shall establish some a-priori estimates on $u(t,\cdot)$.
We will say that the map $u(t,x)$ in~\eqref{solution-u} is an {\it $i$-th simple wave} with profile $\phi^\beta_i$.
We recall that a {\it classical solution} of a Cauchy problem~\eqref{conlaws}, \eqref{indata} 
is a locally Lipschitz continuous map $u: [0,T]\times\R\to \Omega$ that satisfies~\eqref{conlaws}
almost everywhere and~\eqref{indata} for all $x\in\R$. A classical solution of~\eqref{conlaws}, \eqref{indata} 
is in particular an entropy
weak solution of~\eqref{conlaws}, \eqref{indata} (see~\cite[Section~4.1]{Dafermos:Book}).
\begin{lemma}\label{i-simple-wave}
Given $T>0$, $0<d< \overline d$,  $0<b\leq 1/{(2\alpha_1\cdot T)}$, with $\alpha_1$ as in~\eqref{lambdadef}
$($interpreting ${1}/{\alpha_1}\doteq \infty$ when $\alpha_1=0$$)$, 
for any fixed $i=1,\dots,N$, and
for every $\beta\in\mathcal{PC}^1_{[d, b]}$,
the map $u(t,x)$ defined in~\eqref{solution-u} provides a classical solution of
the Cauchy problem
\begin{align}
\label{Cauchy-phi}
&u_t+f(u)_x=0,
\\
\noalign{\smallskip}
\label{Cauchy-phi-in}
&u(0,\cdot)=\phi^\beta_i,
\end{align}
on $[0,T]\times\mathbb{R}$. Moreover, for every $t\leq T$, there hold:
\begin{equation}
\label{est12sw}
\|u(t,\cdot)\|_{L^{\infty}(\R,\Omega)} 
=\|\phi^\beta_i\|_{L^{\infty}(\R,\Omega)}\leq d,
\qquad\quad
\|u_x(t,\cdot)\|_{L^\infty(\R,\Omega)}\leq 2\cdot\Big\|\frac{d}{dx}\phi^\beta_i\Big\|_{L^\infty(\R,\Omega)}
\leq 2b.
\end{equation}
%
\end{lemma}
\begin{proof}
Observe first that, by the definitions~\eqref{inverse-y}, \eqref{solution-u}, 
and because of~\eqref{dphi}, \eqref{d-i-characterbound},
the map $u(t,x)$ is Lipschitz continuous, and it is differentiable at every point $(t,x)$ lying outside the curves 
$t\mapsto (t, x_i(t,y_\ell))$, $\{y_\ell\}_\ell$ denoting the finite collection of points where $\beta$
(and hence
$\phi^\beta_i$) is not differentiable.
Moreover, one has
\begin{equation}
\label{sol-u}
u(t,x_i(t,y))=\phi^\beta_i(y),\qquad\forall (t,y)\in [0,T]\times(\R\setminus\{y_\ell\}_\ell),
\end{equation}
and, recalling~\eqref{Rparam}, the first estimate in~\eqref{est12sw} holds.
Taking the derivative with respect to $t$ and $y$ in both sides of (\ref{sol-u}), 
and recalling~\eqref{i-character}, \eqref{dphi}, we obtain 
\begin{equation}\label{der-u-t}
u_t(t,x_i(t,y))+\lambda_i(u(t,x_i(t,y)))\cdot u_x(t,x_i(t,y))=0,
\end{equation}
and
\begin{equation}
\label{der-u-x}
u_x(t,x_i(t,y))\cdot \frac{\partial}{\partial y}x_i(t,y)
={\beta'}(y)\cdot r_i(u(t,x_i(t,y))),
\end{equation}
at every point $(t,y)\in [0,T]\times(\R\setminus\{y_\ell\}_\ell)$.
We may divide both sides of~\eqref{der-u-x} by $\frac{\partial}{\partial y}x_i(t,y)$
because of~\eqref{d-i-characterbound}, and thus find
\begin{equation}
\label{der-u-x2}
Df(u(t,x_i(t,y)))\cdot u_x(t,x_i(t,y)) = 
\lambda_i(u(t,x_i(t,y)))\cdot u_x(t,x_i(t,y))\,,
\end{equation}
which, together with~\eqref{der-u-t}, yields
\begin{equation}
\label{quasi-eqn}
\nonumber
u_t(t,x)+Df(u(t,x))\cdot u_x(t,x)=0,
\end{equation}
at every point $(t,x)\in ([0,T]\times\R)\setminus \cup_\ell\{(t,x_i(t,y_\ell)\,|\, t\in [0,T]\}$.
On the other hand, since by~\eqref{i-character} $x(0,\cdot)$ is the identity map, 
it follows from~\eqref{inverse-y}, \eqref{solution-u} that  
$u(0,x)=\phi^\beta_i(x)$ for all $x\in\R$.
Therefore, $u(t,x)$ is a Lipschitz continuous map that satisfies the equation~\eqref{Cauchy-phi}
almost everywhere on $[0,T]\times\R$, together with the initial condition~\eqref{Cauchy-phi-in}
at every $x\in\R$. Hence  $u(t,x)$ provides
a classical solution of~\eqref{Cauchy-phi}-\eqref{Cauchy-phi-in}.
Moreover,  relying on~\eqref{normalized-eigenvect},  \eqref{dphi}, 
\eqref{d-i-characterbound}, \eqref{sol-u}, 
and because of the bound on $\dot\beta$ in~\eqref{PCset}, we recover
from~\eqref{der-u-x} the second estimates in~\eqref{est12sw}, thus completing the proof of the lemma.
%
\end{proof}

\subsection{Superposition of simple waves}

We wish to construct now a classical solution of~\eqref{conlaws}, 
on a fixed time interval~$[0,T]$, 
with initial data given by the
profiles of $N$ simple waves, 
one for each characteristic family, supported on disjoint sets.
In order to analyze the behaviour of the solution in the regions of interaction among simple waves
we shall rely on uniform a-priori bounds on a classical solution  $u(t,x)$ of~\eqref{conlaws}
and on its spatial derivative,  which can be derived by a standard technique
(e.g. see~\cite[Section 4.2]{H:Book})
when the initial data has sufficiently small norms $\|u(0,\cdot)\|_{L^{\infty}(\R,\Omega)}$,
$\|u_x(0,\cdot)\|_{L^{\infty}(\R,\Omega)}$. 
In order to state the next lemma that provides such a-priori estimates
we need to introduce some further notation.
Letting $l^T$ denote the transpose (row) vector of a given (column) vector $l\in\R^N$, set
\begin{gather}
\label{alpha-1-2-def1}
\begin{gathered}
\Gamma_2(u)\doteq
\sup_{i,j,k}\bigg\{\big|\lambda_k(u)-\lambda_i(u)\big|\big|l_i^T(u)Dr_j(u)\big|\bigg\},
\\
\noalign{\smallskip}
\Gamma_3(u)\doteq
\sup_{i,j,k} \bigg\{\big|\lambda_k(u)-\lambda_j(u)\big|\big|l_i^T(u)Dr_k(u)\big|\bigg\}
+\sup_i\big|\nabla\lambda_i(u)\big|,
\\
\noalign{\smallskip}
\Gamma_4(u)\doteq 
\sup_i\big|l_i(u)\big|,
\end{gathered}
\\
\noalign{\smallskip}
\label{alpha-1-2-def2}
\alpha_l\doteq \sup\Big\{\Gamma_l(u)\,; \, u\in B_{\overline d}\,\Big\} ,\qquad l=2,3,4.
 \end{gather}
Notice that 
\eqref{normalized-eigenvect} implies $\alpha_4\geq 1$.
%
Comparing~\eqref{lambdadef}, \eqref{alpha-1-2-def1}, \eqref{alpha-1-2-def2}, we deduce that
\begin{equation}
\label{a134}
\alpha_1 \leq  \alpha_3 \leq \alpha_3\alpha_4.
\end{equation}
\begin{lemma}
\label{smooth-bounds}
Given $T>0$, $0<d\leq (\overline d\,e^{-\alpha_2/\alpha_3}) /(2\alpha_4 N)$,  $0<b\leq  1/{(2 \alpha_3\alpha_4 N^2\cdot T)}$,
 with $\alpha_l$, $l=1,2,3$, as in~\eqref{lambdadef} and~\eqref{alpha-1-2-def2}, consider a
 piecewise $C^1$ map $\phi:\R\to\Omega$ 
 that satisfies
 \begin{equation}
 \label{bound-indata}
 \|\phi\|_{L^{\infty}(\R,\Omega)}\leq d,
 \qquad\qquad
\|\phi'\|_{L^{\infty}(\R,\Omega)}\leq b.
 \end{equation}
 Then, the Cauchy problem
\begin{align}
\label{Cauchy-phi2}
&u_t+f(u)_x=0,
\\
\noalign{\smallskip}
\label{Cauchy-phi-in2}
&u(0,\cdot)=\phi,
\end{align}
 admits a classical solution $u(t,x)$ on $[0,T]\times \R$ and, for every $t\leq T$, there hold
 \begin{equation}
 \label{est1-smooth}
 \|u(t,\cdot)\|_{L^{\infty}(\R,\Omega)}\leq 2\alpha_4 N\, e^{\frac{\alpha_2}{\alpha_3}} \cdot d,
 \qquad\qquad
 \|u_x(t,\cdot)\|_{L^{\infty}(\R,\Omega)}\leq 2\alpha_4 N\cdot b.
 \end{equation}
\end{lemma}
\begin{proof}
We provide here only  a sketch of the proof. Further details can be found in~\cite[Section 4.2]{H:Book}.
In order to prove the lemma
it will be sufficient to show that, for every fixed time $T \leq 1/(2\alpha_3\alpha_4 N^2 \cdot b)$, 
and for every initial data $\phi$ satisfying~\eqref{bound-indata}, the estimates~\eqref{est1-smooth} 
hold on $[0, T]$ for a classical solution of~\eqref{Cauchy-phi2}-\eqref{Cauchy-phi-in2}. 
In fact, since 
 by~\eqref{bound-indata} we are assuming the initial  bound
\begin{equation}
\label{phi-bound}
\|\phi\|_{L^{\infty}(\R,\Omega)}\leq d\leq  \frac{\overline d}{2 \,\alpha_4 N}e^{-\frac{\alpha_2}{\alpha_3}} ,
\end{equation}
the first estimate in~\eqref{est1-smooth} 
guarantees in particular that $\|u(t,\cdot)\|_{L^{\infty}(\R,\Omega)}\leq \overline d$
for all $t\in[0, T]$.
As in the proof of~\cite[Theorem 4.2.5]{H:Book},
relying on the a-priori bounds ~\eqref{est1-smooth} one can then actually construct
a classical solution of~\eqref{Cauchy-phi2}-\eqref{Cauchy-phi-in2} on $[0,T]$,
as limit of a Cauchy sequence of approximate solutions of
the linearized problem. 

Thus, assume that $u(t,x)$ is a classical solution of the Cauchy 
problem~\eqref{Cauchy-phi2}-\eqref{Cauchy-phi-in2} on $[0,T]\times \R$,
with
a piecewise $C^1$ initial data $\phi$ satisfying~\eqref{bound-indata}.
We may decompose $u$ and $u_x$ along the basis of right eigenvectors
$r_1(u), \dots, r_N(u)$, writing
\begin{equation}
\label{sol-decomp-1}
u(t,x)=\sum_i p_i(t,x) r_i(u(t,x)),
\qquad\quad
u_x(t,x)=\sum_i q_i(t,x) r_i(u(t,x)),
\end{equation}
which, because of~\eqref{normalized-eigenvect}, is equivalent to set
\begin{equation}
\label{sol-decomp-2}
p_i(t,x)\doteq l_i(u(t,x)) \cdot u(t,x),
\qquad\quad
q_i(t,x)\doteq l_i(u(t,x)) \cdot u_x(t,x),
\qquad i=1,\dots,N.
\end{equation}
Differentiating $p_i, q_i$ along the $i$-th characteristic we find,
for each $i$-th characteristic family, the equations
\begin{equation}
\label{eq-pi-qi}
\begin{aligned}
(p_i)_t+\lambda_i(u(t,x))(p_i)_x
&=\sum_{j,k} \gamma^p_{i,j,k}(u(t,x))\, p_j q_k,
\\
\noalign{\smallskip}
(q_i)_t+\lambda_i(u(t,x))(q_i)_x
&=\sum_{j,k} \gamma^q_{i,j,k}(u(t,x))\, q_j q_k,
\end{aligned}
\end{equation}
where
\begin{equation}
\label{eq-pi-qi-coeff}
\begin{aligned}
\gamma^p_{i,j,k}(u)&\doteq
\big(\lambda_k(u)-\lambda_i(u)\big) l_i^T(u)Dr_j(u) r_k(u),
\\
\noalign{\smallskip}
\gamma^q_{i,j,k}(u)&\doteq
\frac{1}{2}
\big(\lambda_k(u)-\lambda_j(u)\big)l_i^T(u)\big[r_j(u), r_k(u)\big]-
\delta_{i,k} \nabla\lambda_i(u)\cdot r_j(u)
\end{aligned}
\end{equation}
($\delta_{i,k}$ being the Kronecker symbol in~\eqref{normalized-eigenvect} and $[r_j,r_k]$
denoting the Lie bracket of the vector fields $r_j, r_k$).
Observe that, by definitions~\eqref{alpha-1-2-def1}-\eqref{alpha-1-2-def2}, one has
\begin{equation}
\label{eq-pi-qi-coeff-est}
\max_{i,j,k} \big|\gamma^p_{i,j,k}(u)\big|\leq \alpha_2,\qquad\quad
\max_{i,j,k} \big|\gamma^q_{i,j,k}(u)\big|\leq \alpha_3 \qquad
\forall~u\in B_{\overline d}.
\end{equation}
Then, assuming that $\|u(t,\cdot)\|_{L^{\infty}(\R,\Omega)}\leq \overline d$ for all $t\in [0,T]$, it follows from
the second equation in~\eqref{eq-pi-qi}
integrated along the characteristics
that, setting
\begin{equation}
\label{Qdef}
Q(t)\doteq \sum_i \big\|q_i(t,\cdot)\big\|_{L^{\infty}(\R,\Omega)},
\end{equation}
there holds
\begin{equation}
\label{est-Q}
Q(t)\leq Q(0) + \alpha_3 N \int_0^t (Q(s))^2ds
\qquad\forall~t.
\end{equation}
By a comparison argument one then derives from~\eqref{est-Q}
that  
\begin{equation}
\label{est-Q-2}
Q(t)\leq \frac{Q(0)}{1 - \alpha_3 N t\,Q(0)}
\qquad\quad\forall t \in \left[0, \frac{1}{\,\alpha_3 N\,Q(0)} \right[.
\end{equation}
On the other hand, notice that by~\eqref{normalized-eigenvect}, \eqref{alpha-1-2-def2}, \eqref{sol-decomp-1}, \eqref{sol-decomp-2},
and recalling~\eqref{normalized-eigenvect}, 
one has
\begin{equation}
\label{uxQ}
 \|u_x(t,\cdot)\|_{L^{\infty}(\R,\Omega)} \leq Q(t) \leq \alpha_4 N \|u_x(t,\cdot)\|_{L^{\infty}(\R,\Omega)}.
\end{equation}
Since we assume by~\eqref{bound-indata} the initial bound
\begin{equation}
\label{phix-bound}
\|\phi'\|_{L^{\infty}(\R,\Omega)}\leq b\leq  \frac{1}{2 \,\alpha_3\,\alpha_4 N^2\cdot T},
\end{equation}
which, in turn, because of~\eqref{uxQ} implies 
\begin{equation*}
Q(0)\leq  \frac{1}{2\,\alpha_3 N\,\cdot T},
\end{equation*}
we obtain
\begin{equation*}
Q(t) \leq 2 Q(0) \qquad\quad \forall t \in [0,T].
\end{equation*}
We deduce with~\eqref{est-Q-2}, \eqref{uxQ}, that
\begin{equation}
\label{est-Q-3}
 \|u_x(t,\cdot)\|_{L^{\infty}(\R,\Omega)}\leq Q(t)\leq 2 Q(0) \leq 2\alpha_4 N\cdot\|\phi'\|_{L^{\infty}(\R,\Omega)}
 \qquad\forall~t\leq T,
\end{equation}
proving the second inequality in~\eqref{est1-smooth}. 
Next, setting
\begin{equation}
\label{Pdef}
P(t)\doteq \sum_i \big\|p_i(t,\cdot)\big\|_{L^{\infty}(\R,\Omega)},
\end{equation}
and integrating the first equation in~\eqref{eq-pi-qi} along the characteristic, we derive
\begin{equation}
\label{est-P}
P(t)\leq P(0) + \alpha_2 N \int_0^t P(s)Q(s) ds
\qquad\forall~t.
\end{equation}
Then, applying Gronwall's lemma, we deduce from~\eqref{est-P}
that
\begin{equation}
\label{est-P2}
P(t)\leq P(0) \exp\bigg( \alpha_2 N \int_0^t  Q(s) ds\bigg)
\qquad\forall~t.
\end{equation}
On the other hand observe that since~\eqref{phix-bound} implies $Q(0)\leq  \frac{1}{2\,\alpha_3 N\,\cdot t}$
for all $t\leq T$,
we deduce from~\eqref{est-Q-2} that 
\begin{equation}
\label{Qest-3}
\int_0^t Q(s) ds  \leq  2 Q(0)\,t  \leq  \frac{1}{\alpha_3 N}
\qquad\forall~t\leq T.
\end{equation}
Moreover, by~\eqref{normalized-eigenvect}, \eqref{alpha-1-2-def2}, \eqref{sol-decomp-1}, \eqref{sol-decomp-2}
there holds
\begin{equation}
\label{uP}
\|u(t,\cdot)\|_{L^{\infty}(\R,\Omega)}  \leq P(t)  \leq  \alpha_4 N \|u(t,\cdot)\|_{L^{\infty}(\R,\Omega)}.
\end{equation}
Hence, \eqref{est-P2}-\eqref{uP} together yield
\begin{equation}
\label{est-P3}
\|u(t,\cdot)\|_{L^{\infty}(\R,\Omega)}  \leq P(t)
\leq P(0) e^{\frac{\alpha_2}{\alpha_3}}
\leq \alpha_4 N \, e^{\frac{\alpha_2}{\alpha_3}}\|\phi\|_{L^{\infty}(\R,\Omega)}
\qquad\forall~t\leq T.
\end{equation}
 This completes the proof of the first inequality in~\eqref{est1-smooth},
 and hence of the lemma.
\end{proof}
Relying on Lemma~\ref{i-simple-wave} and Lemma~\ref{smooth-bounds}
we shall construct now a classical solution $u(t,x)$ of~\eqref{conlaws}
on a given time interval $[0,T]$,
so that:
\vspace{-2pt}
\begin{itemize}
\item[-] the initial data $u(0,\cdot)$ is supported on $N$ disjoint intervals $I_i$, $i=1,\dots,N$, of the same length $|I_i|=L$,
and on each interval $I_i$ it coincides with the profile of a simple wave of the $i$-th characteristic family;
\item[-] the terminal value $u(T,\cdot)$ at time $T$ is supported on an interval of length $\approx 2L$.
\end{itemize}
Namely, given $L, b>0$, $0<d< \overline d$ and
%
\begin{equation}
\label{T-assumption1}
T\geq \frac{L}{\Delta_{\wedge}\lambda},\qquad
\Delta_{\wedge}\lambda\doteq \displaystyle{\min_i}\big\{\lambda_{i+1}(0)-\lambda_i(0)\big\},
\end{equation}
%
%
set
\begin{equation}
\label{Itauset-def2}
\xi_i^-
\doteq -L/2 -\lambda_{i}(0)\cdot T,
\qquad\quad
\xi_i^+
\doteq \xi_i^-+L,
\qquad\quad i=1,\dots,N,
\end{equation}
%
and consider the family of $N$-tuples of maps
\begin{equation}
\label{PCNset}
\mathcal{PC}^{1,N}_{[L,d,b,T]}\doteq
\Big\{
\beta=(\beta_1,\dots,\beta_N)\in(\mathcal{PC}^1_{[d,b]})^{N}\ \big| \ \mathrm{Supp}(\beta_i)\subset [\xi_i^-, \xi_i^+],
\ i=1,\dots,N
\Big\},
\end{equation}
%
where $\mathcal{PC}^1_{[d,b]}$ denotes the class of functions introduced in~\eqref{PCset}.
Observe that, by~\eqref{T-assumption1}, \eqref{Itauset-def2}, one
has
\begin{equation}
\label{xiest1}
\xi_{i+1}^+\leq \xi_i^-\qquad\forall~i=1,\dots, N-1.
\end{equation}
Then,
let $\beta=(\beta_1,\dots,\beta_N)\in\mathcal{PC}^{1,N}_{[L,d,b,T]}$, 
and define  the function
$\phi^\beta:\R\to \Omega$, by setting
\begin{align}
\label{phidef}
\phi^\beta(x)&\doteq \sum_i \phi^{\beta_i}_i(x)=
\begin{cases}
\phi^{\beta_i}_i(x)\ \ &\text{if}\quad \ x\in \mathrm{Supp}(\beta_i), \ i=1,\dots,N,
\\
\noalign{\smallskip}
\ \ 0&\text{otherwise,}
\end{cases}
\end{align}
where 
\begin{equation}
\label{fibetadef2}
\phi^{\beta_i}_i(x)\doteq R_i(\beta_i(x))
\end{equation}
denotes a map defined as in~\eqref{j-initial wave}
in connection with $\beta_i\in\mathcal{PC}^1_{[d,b]}$.
The next Lemma shows that if we also assume 
\begin{equation}
\label{db-assumption2}
0 <  d  \leq  \frac{1}{2 \alpha_4 N e^{\alpha_2/\alpha_3}} \cdot
\min\bigg\{\overline d,\, \frac{\Delta_{\wedge}\lambda}{2\alpha_1}\bigg\},
\qquad\quad
0<b\leq 
\min\bigg\{\frac{1}{2\alpha_1\cdot T},\, 
\frac{\Delta_{\wedge}\lambda}{4 \alpha_3 \alpha_4 N^2\cdot L}
\bigg\},
\end{equation}
$($interpreting 
${1}/{\alpha_1}\doteq \infty$
when $\alpha_1=0$$)$, 
for every given $\beta\in\mathcal{PC}^{1,N}_{[L,d,b,T]}$
we can  apply Lemma~\ref{i-simple-wave} and Lemma~\ref{smooth-bounds}
to derive the existence of a classical solution of~\eqref{conlaws}
with initial data~$\phi^\beta$
which possesses the desired properties. 
\begin{proposition}
\label{n-simple-waves-overlap}
Let $f:\Omega \to \R^N$ be a $C^2$ map defined on an open, connected domain $\Omega\subset\R^N$,
$\Omega\supset B_{\overline d}\doteq \{u\in\R^n\, |\, |u|\leq\overline d\,\}$,
and assume that the Jacobian matrix $Df(u)$ has $N$ real, distinct 
 eigenvalues $\lambda_1(u)<...<\lambda_N(u)$.
Given $L,T, d,b>0$, satisfying~\eqref{T-assumption1}, \eqref{db-assumption2}
$($with $\alpha_1$ as in~\eqref{lambdadef}, 
$\alpha_l$, $l=2,3,4$, as in~\eqref{alpha-1-2-def2},
and $\Delta_{\wedge}\lambda$ as in~\eqref{T-assumption1}$)$,
let $\mathcal{PC}^{1,N}_{[L,d,b,T]}$ be the class of maps introduced in~\eqref{Itauset-def2}-\eqref{PCNset},
and consider a map $\phi^\beta:\R\to \Omega$  as in~\eqref{phidef}, defined in connection with
an $N$-tuple $\beta=(\beta_1,\dots,\beta_N)\in \mathcal{PC}^{1,N}_{[L,d,b,T]}$.
Then, 
there exists a classical solution $u(t,x)$ of
the Cauchy problem
\begin{align}
\label{Cauchy-phi-n}
&u_t+f(u)_x=0,
\\
\noalign{\smallskip}
\label{Cauchy-phi-in-n}
&u(0,\cdot)=\phi^\beta,
\end{align}
on $[0,T]\times\mathbb{R}$. Moreover, setting
\begin{equation}
\label{alfa4def}
\alpha_5\doteq
\frac{\lambda_N(0)-\lambda_1(0)}{\Delta_{\wedge}\lambda},
\end{equation}
one has
\begin{equation}
\label{supp-w}
\mathrm{Supp}(u(T,\cdot))\subseteq [-L\cdot(1+\alpha_5),\, L\cdot(1+\alpha_5)],
\end{equation}
%
and, for every $t\leq T$, there hold:
\begin{gather}
\label{est1swn}
\|u(t,\cdot)\|_{L^{\infty}(\R,\Omega)} 
\leq 2\alpha_4 N\, e^{\frac{\alpha_2}{\alpha_3}} \cdot d,
\qquad\
\|u_x(t,\cdot)\|_{L^\infty(\R,\Omega)}\leq 4\alpha_4 N\cdot b.
%
\end{gather}
\end{proposition}
\begin{proof}
We will prove the existence of a classical solution of the Cauchy problem~\eqref{Cauchy-phi-n}-\eqref{Cauchy-phi-in-n}
on $[0,T]$ satisfying~\eqref{supp-w}-\eqref{est1swn},  
by first showing that such a solution is obtained on $[0, T-L/\Delta_{\wedge}\lambda]$ as a superposition of simple waves
supported on disjoint set, and next deriving a-priori bounds on the solution and its support in the interval $[T-L/\Delta_{\wedge}\lambda,\,T]$.

\noindent{\bf 1.} 
Given $\beta=(\beta_1,\dots,\beta_N)\in \mathcal{PC}^{1,N}_{[L,d,b,T]}$,
define as in~\eqref{i-character} the functions
\begin{equation}
\label{i-char2}
x^\flat_i(t,y)\doteq y+\lambda_i(\phi^{\beta_i}_i(y))\cdot t,
\qquad t\geq 0,
\end{equation}
for each $i=1,\dots,N$.
Since ~\eqref{db-assumption2} implies $t \leq  1/(2\alpha_1\cdot b)$
for all $t\in [0,T]$, by the inequality in~\eqref{d-i-characterbound} we
deduce 
that the maps $y\mapsto x^\flat_i(t,y)$, $i=1,\dots,N$,  are
one-to-one in $\R$, for every fixed $t\in [0,T]$.
Then, setting
\begin{equation}
\label{inverse-yi}
z^\flat_i(t,\cdot)\doteq\ (x^\flat_i)^{\!^{-1}}\!(t,\cdot),
\qquad i=1,\dots,N,
\end{equation}
and letting $\phi^{\beta_i}_i$ be the map in~\eqref{fibetadef2}, define the function
\begin{equation}
\label{solution-ui}
u^\flat(t,x)\doteq 
\begin{cases}
\phi^{\beta_i}_i(z^\flat_i(t,x))\ \ &\text{if}\quad \ x\in [x^\flat_i(t,\xi_i^-),\, x^\flat_i(t,\xi_i^+)]\setminus
\displaystyle{\bigcup_{j\neq i}}[x^\flat_j(t,\xi_j^-),\, x^\flat_j(t,\xi_j^+)], \ \ \ i=1,\dots,N,
\\
\noalign{\smallskip}
\ \ 0&\text{otherwise,}
\end{cases}
\end{equation}
on $[0, T] \times\R$.
 Observe that, because of~\eqref{PCNset}, \eqref{fibetadef2}, one has $\phi^{\beta_i}_i(\xi_i^\pm)=R_i(0)=0$,
for all $i=1,\dots,N$.
Hence, recalling~\eqref{T-assumption1}, \eqref{Itauset-def2}, and by~\eqref{i-char2}, there holds
\begin{equation}
\label{chardist2}
  x^\flat_{i+1}(t,\xi_{i+1}^\pm) \leq x^\flat_i(t,\xi_i^\pm)
  \qquad\forall~t\in \bigg[0, T-\frac{L}{\Delta_{\wedge}\lambda}
  \bigg],
  \qquad i=1,\dots, N-1,
\end{equation}
so that one has
\begin{equation}
\label{solution-ui2}
u^\flat(t,x)=
\begin{cases}
\phi^{\beta_i}_i(z^\flat_i(t,x))\ \ &\text{if}\quad \ x\in [x^\flat_i(t,\xi_i^-),\, x^\flat_i(t,\xi_i^+)],\ \ i=1,\dots,N,
\\
\noalign{\smallskip}
\ \ 0&\text{otherwise,}
\end{cases}
\end{equation}
for all $(t,x)\in [0, T-L/\Delta_{\wedge}\lambda]
\times\R$.

By~\eqref{solution-ui2} the restriction of $u^\flat(t,x)$ to the domain 
$[0, T-L/\Delta_{\wedge}\lambda
]\times\R$
is a Lipschitz continuous map supported
on the disjoint union of  sets 
\begin{equation}
\label{Ddef}
D_i\doteq \big\{(t,x)\ | \  t\in [0,T-L/\Delta_{\wedge}\lambda], \ 
x\in [x^\flat_i(t,\xi_i^-), x^\flat_i(t,\xi_i^+)]\big\},
\qquad i=1,\dots, N.
\end{equation}
Since 
\eqref{db-assumption2} implies $b<1/(2\alpha_1\cdot T)$,
we know by Lemma~\ref{i-simple-wave} that $u^\flat(t,x)$ is a classical solution of~\eqref{Cauchy-phi-n}
 on each set $D_i$. Moreover, recalling that $z^\flat_i(0,\cdot)$ is the identity map,
 by~\eqref{solution-ui2} one has $u^\flat(0,x)=\phi^{\beta_i}_i(x)$, for all $x\in[\xi_i^-, \xi_i^+]$, $i=1,\dots, N$.
Therefore, looking at~\eqref{PCNset}, \eqref{phidef}, we deduce that~\eqref{Cauchy-phi-in-n} holds. Hence,
it follows that $u^\flat(t,x)$ provides a classical solution of~\eqref{Cauchy-phi-n}-\eqref{Cauchy-phi-in-n} 
on $[0, T-L/\Delta_{\wedge}\lambda]\times\mathbb{R}$. 
\smallskip

Notice that, 
letting $u^\flat(t,\cdot)_{|D_i(t)}$ denote the restriction of $u^\flat(t,\cdot)$ to the set 
$D_i(t)\doteq [x^\flat_i(t,\xi_i^-), x^\flat_i(t,\xi_i^+)]$,
we deduce from \eqref{phidef}, \eqref{solution-ui2}, that for every 
$t\in [0, T-L/\Delta_{\wedge}\lambda]$ there holds
\begin{equation}
\label{swn-infnorm}
\begin{aligned}
\|u^\flat(t,\cdot)\|_{L^{\infty}} 
&= \max_i \|u^\flat(t,\cdot)_{|D_i(t)}\|_{L^{\infty}},
\qquad\quad
\|u^\flat_x(t,\cdot)\|_{L^{\infty}} 
= \max_i \|u^\flat(t,\cdot)_{|D_i(t)}\|_{L^{\infty}},
\\
\noalign{\smallskip}
\|\phi^\beta\|_{L^{\infty}} &= \max_i \|\phi_i^{\beta_i}\|_{L^{\infty}},
\qquad\qquad\qquad\qquad\
\Big\|\frac{d}{dx}\phi^\beta\Big\|_{L^{\infty}} = \max_i \Big\|\frac{d}{dx}\phi_i^{\beta_i}\Big\|_{L^{\infty}}.
\end{aligned}
\end{equation}
Therefore, relying on the estimate~\eqref{est12sw}
for each $u^\flat(t,\cdot)_{|D_i(t)}$,
we derive from~\eqref{swn-infnorm} the estimates
\begin{equation}
\label{est3swn}
\|u^\flat(t,\cdot)\|_{L^{\infty}(\R,\Omega)} 
=\|\phi^\beta\|_{L^{\infty}(\R,\Omega)}\leq d,
\qquad\quad
\|u^\flat_x(t,\cdot)\|_{L^{\infty}(\R,\Omega)} 
\leq 2\Big\|\frac{d}{dx}\phi^\beta\Big\|_{L^{\infty}(\R,\Omega)}\leq 2b,
\end{equation}
for all $t\in [0, T-L/\Delta_{\wedge}\lambda]$.
\smallskip

\noindent
{\bf 2.}
Observe now that
%
\begin{equation}
\phi(x)\doteq u^\flat(T-L/\Delta_{\wedge}\lambda,\,x), \qquad x\in\R,
\end{equation}
is a piecewise $C^1$ map that satisfies the estimates~\eqref{est3swn}, 
%
with $d,b$ verifying the 
bounds~\eqref{db-assumption2}.
Thus, applying Lemma~\ref{smooth-bounds} we deduce  the
existence of a classical solution $u^\sharp(t,x)$ 
of~\eqref{Cauchy-phi-n} on the domain $[T-L/\Delta_{\wedge}\lambda,\, T]\times\R$,
that assumes the initial data
$u^\sharp(T-L/\Delta_{\wedge}\lambda,\,\cdot)=\phi$,
at time $t=T-L/\Delta_{\wedge}\lambda$.
Moreover, by~\eqref{est1-smooth}, \eqref{est3swn}, there holds
%
  %
 \begin{equation}
 \label{est2-smooth}
 \|u^\sharp(t,\cdot)\|_{L^{\infty}(\R,\Omega)}\leq 2\alpha_4 N\, e^{\frac{\alpha_2}{\alpha_3}} \cdot d,
 \qquad\qquad
 \|u^\sharp_x(t,\cdot)\|_{L^{\infty}(\R,\Omega)}\leq 4\alpha_4 N\cdot b,
 \end{equation}
for all $t\in [T-L/\Delta_{\wedge}\lambda,\, T]$.
Therefore, the function defined by
\begin{equation}
\label{u-ufl-ush}
u(t,x)\doteq
\begin{cases}
u^\flat(t,x)\quad &\text{if}\qquad t\in [0,\, T-L/\Delta_{\wedge}\lambda],
\\
\noalign{\smallskip}
u^\sharp(t,x)\quad &\text{if}\qquad t\in \,]T-L/\Delta_{\wedge}\lambda, \, T],
\end{cases}
\end{equation}
provides a classical solution of~\eqref{Cauchy-phi-n},~\eqref{Cauchy-phi-in-n} that, 
because of~\eqref{est3swn}, \eqref{est2-smooth},  
satisfies the bounds~\eqref{est1swn} for all $t\in [0,T]$. 

To conclude the proof of the proposition we shall
derive now an estimate of the support of $u(T,x)$. Consider, for each $i$-th family,
the $i$-th characteristic curve
of $u$ through a point $(\tau,y)\in [0,T]\times\R$, denoted by $t\mapsto x_i(t;\, \tau, y)$, $t\in [0,T]$,
and defined as the (unique) solution of the Cauchy problem
\begin{equation}
\label{char-u-def}
\dot x = \lambda_i(u(t,x)),
\qquad\quad x(\tau)=y.
\end{equation}
Set, for every $i=1,\dots,N$, 
\begin{equation}
\label{tau-i-def}
\begin{aligned}
\tau_i^-&\doteq \inf\big\{t\in[0,T]\, ; \, x_i(t; 0, \xi_i^-)=x_j(t; 0, \xi_j^\pm)\ \  \text{for some}\ \ j\neq i\big\},
\qquad\
y_i^-\doteq  x_i(\tau_i^-; 0, \xi_i^-),
\\
\noalign{\smallskip}
\tau_i^+&\doteq \inf\big\{t\in[0,T]\, ; \, x_i(t; 0, \xi_i^+)=x_j(t; 0, \xi_j^\pm)\ \  \text{for some}\ \ j\neq i\big\},
\qquad\
y_i^+\doteq  x_i(\tau_i^+; 0, \xi_i^+),
\end{aligned}
\end{equation}
where the equality  $x_i(t; 0, \xi_i^-)=x_j(t; 0, \xi_j^\pm)$ is interpreted as
$x_i(t; 0, \xi_i^-)=x_j(t; 0, \xi_j^-)$ or $x_i(t; 0, \xi_i^-)=x_j(t; 0, \xi_j^+)$, and analogously
for $x_i(t; 0, \xi_i^+)=x_j(t; 0, \xi_j^\pm)$.
Next, consider the union of the regions confined between the minimal and maximal characteristics emanating 
from the points $(\tau_i^\pm, y_i^\pm)$, $i=1,\dots,N$:
\begin{equation}
\label{Lambda-def}
\Lambda\doteq 
\bigcup_i \big(\Lambda_i^-\cup  \Lambda_i^+\big),\qquad
\begin{aligned}
\Lambda_i^-&\doteq 
\Big\{(t,x)\in[\tau_i^-,T]\times\R
\, ; \, x_1(t; \tau_i^-, y_i^-)\leq x \leq x_N(t; \tau_i^-, y_i^-)
\Big\},
\\
\noalign{\smallskip}
\Lambda_i^+&\doteq 
\Big\{(t,x)\in[\tau_i^+,T]\times\R
\, ; \, x_1(t; \tau_i^+, y_i^+)\leq x \leq x_N(t; \tau_i^+, y_i^+)
\Big\}.
\end{aligned}
\end{equation}

\begin{figure}[htbp] \label{fig:Lambdas}
\begin{center}
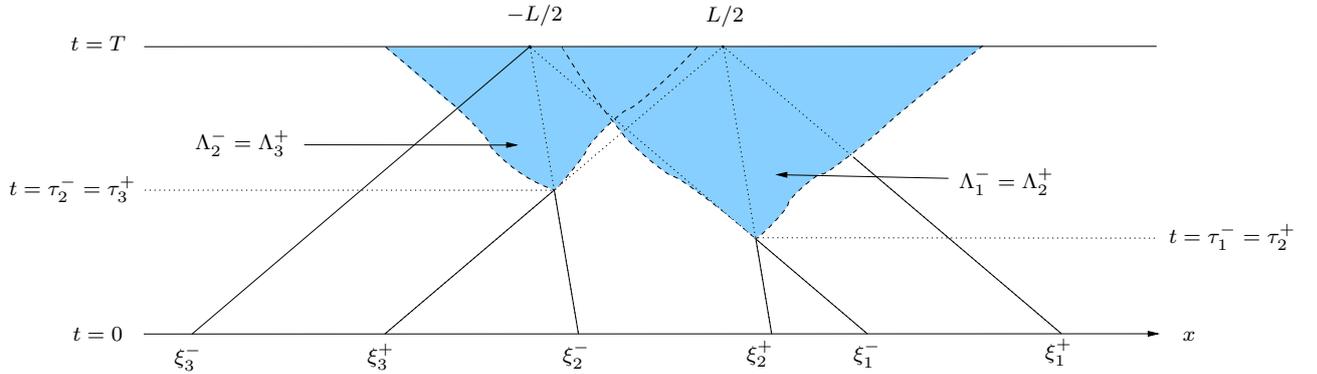
\end{center}
\caption{The sets $\Lambda_{i}^{\pm}$}
\end{figure}
Observe that $([0,T]\times \R)\setminus\Lambda$ is a domain of determinacy for 
the Cauchy problem~\eqref{Cauchy-phi-n}-\eqref{Cauchy-phi-in-n}, since, for every
fixed $(\tau, y)\in ([0,T]\times \R)\setminus\Lambda$ and for any $i=1,\dots,N$, one has 
$\{(t,x_i(t;\,\tau,y))\,; 0\leq t \leq \tau\}\subset ([0,T]\times \R)\setminus\Lambda$.
Therefore, we deduce that the classical solution $u(t,x)$ of~\eqref{Cauchy-phi-n},~\eqref{Cauchy-phi-in-n}
coincides with the function $u^\flat(t,x)$ defined
in~\eqref{solution-ui} on the whole region $([0,T]\times \R)\setminus\Lambda$,
and that there hold
\begin{equation}
\label{solution-ui3}
u(t,x)=
\begin{cases}
\phi^{\beta_i}_i(z^\flat_i(t,x))\ \ &\text{if}\quad \ x\in [x^\flat_i(t,\xi_i^-),\, x^\flat_i(t,\xi_i^+)],\ \ i=1,\dots,N,
\\
\noalign{\smallskip}
\ \ 0&\text{otherwise,}
\end{cases}
\end{equation}
for all $(t,x)\in ([0,T]\times \R)\setminus\Lambda$, and 
\begin{equation}
\label{i-char3}
x_i(t;\,0,\xi_i^\pm)=x^\flat_i(t,\xi_i^\pm)
\qquad\forall~t\in [0,\tau_i^\pm],\qquad i=1,\dots,N,
\end{equation}
with obvious meaning of notations.
Notice that, since by~\eqref{PCNset}, \eqref{fibetadef2} one has $\phi^{\beta_i}_i(\xi_i^\pm)=R_i(0)=0$,
it follows from~\eqref{Itauset-def2}, \eqref{i-char2} that 
\begin{equation}
\label{char-i-T}
x^\flat_i(T,\xi_i^\pm)=\pm L/2.
\end{equation}
Thus, letting $u(T,\cdot)_{|D}$ denote the restriction of $u(T,\cdot)$ to a set $D$, we deduce from~\eqref{solution-ui3} that
\begin{equation}
\label{supp-est1}
\mathrm{Supp}\Big(u(T,\cdot)_{\big|\{x;\, (T,x)\notin\Lambda\}}\Big)\subseteq [-L/2, L/2].
\end{equation}
On the other hand, observe that by~\eqref{chardist2}, \eqref{tau-i-def}, \eqref{i-char3}, one has
\begin{equation}
\label{bound-taui}
\inf\big\{\tau_i^-, \tau_i^+;\, i=1,\dots,N\big\}\geq T-\frac{L}{\Delta_{\wedge}\lambda}.
\end{equation}
Moreover, the first estimate in~\eqref{est1swn},
together with the bound~\eqref{db-assumption2}, imply in particular $\|u(t,\cdot)\|_{L^{\infty}} < \overline d$,
for all $t\in [0,T]$, while~\eqref{PCNset}, \eqref{fibetadef2}, \eqref{tau-i-def}, \eqref{solution-ui3} yield
\begin{equation}
\label{uchar-i-T}
u(\tau_i^-, y_i^-)=\phi^{\beta_i}_i(\xi_i^-)=0,\qquad\qquad
u(\tau_i^+, y_i^+)=\phi^{\beta_i}_i(\xi_i^+)=0.
\end{equation}
Thus, relying on~\eqref{lambdadef}, \eqref{db-assumption2}, \eqref{est1swn}, \eqref{i-char2},
\eqref{char-u-def}, \eqref{i-char3}, \eqref{char-i-T}, \eqref{bound-taui}, \eqref{uchar-i-T} we derive
\begin{equation}
\label{supp-est2}
\begin{aligned}
x_N(T; \tau_i^\pm, y_i^\pm)
&\leq y_i^\pm +\bigg(\lambda_N(0)+2 \alpha_1\,\alpha_4 N\, e^{\frac{\alpha_2}{\alpha_3}} \cdot d
\bigg)\cdot(T-\tau_i^\pm)
\\
\noalign{\smallskip}
&=x^\flat_i(\tau_i^\pm,\,\xi_i^\pm) +
\bigg(\lambda_N(0)+2 \alpha_1\,\alpha_4 N\, e^{\frac{\alpha_2}{\alpha_3}}\cdot d
\bigg)\cdot(T-\tau_i^\pm)
\\
\noalign{\smallskip}
&\leq x^\flat_i(T,\,\xi_i^\pm) + 
\bigg((\lambda_N(0)-\lambda_i(0))+2 \alpha_1\,\alpha_4 N\, e^{\frac{\alpha_2}{\alpha_3}}\cdot d
\bigg)
\cdot \frac{L}{\Delta_{\wedge}\lambda}
\\
\noalign{\smallskip}
&\leq L\cdot\Bigg(\frac{1}{2}+\frac{2 \alpha_1\,\alpha_4 N\, e^{\frac{\alpha_2}{\alpha_3}}}{\Delta_{\wedge}\lambda}
\cdot d+\frac{\lambda_N(0)-\lambda_1(0)}{\Delta_{\wedge}\lambda}\Bigg)
\\
\noalign{\smallskip}
&\leq L\cdot\Bigg(1+\frac{\lambda_N(0)-\lambda_1(0)}{\Delta_{\wedge}\lambda}\Bigg),
\end{aligned}
\end{equation}
and, analogously,
\begin{equation}
\label{supp-est3}
x_1(T; \tau_i^\pm, y_i^\pm)
\geq -L\cdot\Bigg(1+\frac{\lambda_N(0)-\lambda_1(0)}{\Delta_{\wedge}\lambda}\Bigg).
\end{equation}
Then, recalling~\eqref{alfa4def} and looking at the definition~\eqref{Lambda-def}
of $\Lambda$, we deduce from~\eqref{supp-est2}-\eqref{supp-est3} that there holds
\begin{equation}
\label{supp-est4}
\mathrm{Supp}\Big(u(T,\cdot)_{\big|\{x;\, (T,x)\in\Lambda\}}\Big)\subseteq [-L\cdot(1+\alpha_5),\, L\cdot(1+\alpha_5)].
\end{equation}
In turn, the inclusion \eqref{supp-est4} together with~\eqref{supp-est1} yields~\eqref{supp-w}, 
completing the proof of the proposition.
\end{proof}
\begin{remark} 
\label{uniqueness}
Classical solutions of conservation laws coincide  with the trajectory of 
the corresponding semigroup, whenever their initial data belongs to the domain
of the semigroup. In fact, by the result in~\cite[Section 5.3]{Dafermos:Book},
if~\eqref{conlaws}-\eqref{indata} admits a classical solution, then such a solution coincide with
any entropy weak solution of the same Cauchy problem.
Therefore, if we consider a general system of conservation laws that generates a
semigroup $(S_t)_{t\geq 0}$ of entropy weak solutions with a domain $\mathcal{D}_0$ as in~\eqref{domainsgr},
and we suppose that the map
$\phi^\beta$ defined in
\eqref{phidef}
satisfies
$\mathrm{Tot. var. }(\phi^\beta) \leq\delta_0$, 
it follows that the classical solution $u(t,\cdot)$ of the Cauchy problem~\eqref{Cauchy-phi-n}-\eqref{Cauchy-phi-in-n}
provided by Proposition~\ref{n-simple-waves-overlap}
coincides with 
$S_t \phi^\beta$.
\end{remark}

\subsection{Simple waves for rich systems}

Here we analyze the structure of simple waves for a class of systems, the so-called {\it rich systems}, 
that can be put in diagonal form with respect to Riemann coordinates.
We recall that a system of conservation laws~\eqref{conlaws} is called a {\it rich system} (see \cite{serre})
if there exists a set of coordinates $w=(w_1,\ldots,w_N)$ consisting of Riemann invariants $w_i=W_i(u)$,
$u\in \Omega$,
associated to each characteristic field $r_i$. 
It is not restrictive to
assume that the Riemann coordinates
are chosen so that  $W(0)=0$.
A necessary and sufficient condition for 
the existence of Riemann coordinates is the Frobenius involutive relation
$[r_i, r_j]=\alpha_i^j r_i + \alpha_j^i r_j$, that must be satisfied,
for some scalar functions $\alpha_i^j, \alpha_j^i$, for all $i,j=1,\dots,N$.
When a system is endowed with a coordinate system of Riemann invariants
it is convenient to normalize the eigenvectors $r_1,\dots, r_N$  of $Df$
so that there holds
\begin{equation}
\label{normalized-2-eigenvect}
\nabla W_i\cdot r_j \equiv\delta_{i,j}
\end{equation}
instead of $|r_i|\equiv 1$ as in~\eqref{normalized-eigenvect}. In turn, \eqref{normalized-2-eigenvect}
implies (cf.~\cite[Section~7.3]{Dafermos:Book}):
\begin{equation}
\label{commute-fields}
[r_i,r_j]\equiv 0\qquad \forall~i,j=1,\dots,N. 
\end{equation}
Throughout the following, we will write
$w_i(t,x)\doteq W_i\big(u(t,x)\big)$ to denote the $i$-th
Riemann coordinate of a solution $u=u(t,x)$ to~\eqref{conlaws},
and we shall adopt the norms $\|w\|_{L^1}\doteq \sum_i \|w_i\|_{L^1}$,
$\|w\|_{L^\infty}\doteq \max_i \|w_i\|_{L^\infty}$.
Notice that, because of~\eqref{normalized-2-eigenvect}, multiplying \eqref{conlaws} from the
left by $DW_i$, $i=1,\dots, N$,
we deduce that the system~\eqref{conlaws} is equivalent to the
system in diagonal form
\begin{equation}
\label{sys-diag}
(w_i)_t+ \lambda_i(w) (w_i)_x=0,
\qquad\quad i=1,\dots,N,
\end{equation}
within the context of classical solutions. 
Thus, letting 
$t\mapsto x_i(t,y)$ denote the $i$-th characteristic of~\eqref{sys-diag}
starting at $y\in\R$, i.e. the solution of the Cauchy problem
\begin{equation}
\label{char-w-def}
\dot x = \lambda_i(w(t,x)),
\qquad\quad x(0)=y,
\end{equation}
it follows that each $i$-th
Riemann coordinate  $w_i(t,x)$
of a  classical solution to~\eqref{conlaws} remains constant along every $i$-th characteristic
of~\eqref{sys-diag}. On the other hand, differentiating~\eqref{sys-diag} w.r.t. $x$, 
and setting $q_i(t,x)\doteq (w_i(t,x))_x$,
we find that
\begin{equation}
\label{eq-qi}
(q_i)_t+\lambda_i(w(t,x))(q_i)_x
=-\sum_j \frac{\partial}{\partial w_j} \lambda_i(w(t,x)) \, q_j q_i.
\end{equation}
Observe that, by virtue of~\eqref{normalized-2-eigenvect}, the inverse map $u=W^{-1}(w)$
of $w=W(u)= (W_1(u), \dots, W_N(u))$
satisfies
$\partial u(w)/ \partial w_i=r_i(u(w))$, for all $i=1,\dots,N$,
and so the chain rule yields
\begin{equation}
\label{eigenval-deriv-riem}
\frac{\partial}{\partial w_j} \lambda_i(w)\Bigg|_{w=W(u)}=
\nabla \lambda_i(u)\cdot r_j(u)\qquad\quad\forall~i,j\,.
\end{equation}
Next, set
\begin{align}
\label{alpha-2w-def1a}
\alpha_1'
&\doteq
\sup\bigg\{\big|\nabla\lambda_i(u)\cdot r_i(u)\big|\ ; \ u\in B_{\overline d}, \ i=1,\dots,N\bigg\},
\\
\label{alpha-2w-def1}
\alpha_1''
&\doteq
\sup\bigg\{
\big|\nabla\lambda_i(u)\cdot r_j(u)\big|\ ; \ u\in B_{\overline d}, \ i,j=1,\dots,N\bigg\},
\end{align}
where $B_{\overline d}$ denotes as usual a ball centered in the origin and contained in the domain $\Omega$ of the flux
function~$f$.
Since $W(0)=0$, we may assume that 
\begin{equation}
\label{w-d-ball}
\Big\{W^{-1}(w)\, | \ |w_i|\leq \overline d'\Big\} \subset B_{\overline d},
\end{equation}
for some $\overline d'>0$.
Thus, because of~\eqref{eigenval-deriv-riem}, \eqref{alpha-2w-def1}, we have
\begin{equation}
\label{bound-grad-lambda-w}
\big|\nabla \lambda_i(w)\big|\leq \sqrt N\,\overline\alpha''_1\qquad\quad\forall~w\in \left[-\overline d',\, \overline d'\,\right]^N\,,
\qquad i=1,\dots,N\,.
\end{equation}
%
Then, with the same arguments of the proof of Lemma~\ref{smooth-bounds},
we deduce the following sharper a-priori bounds on the Riemann coordinate expression of a 
classical solution of a rich system of conservation laws. 
\begin{lemma}
\label{smooth-bounds-rich}
Assume that~\eqref{Cauchy-phi2} is a strictly hyperbolic
and rich system.
Given $T>0$, $0<d\leq \overline d'$,  $0<b\leq 1/{( 2 \alpha_1'' N \cdot T)}$,
with $\alpha_1''$
as in~\eqref{alpha-2w-def1} $($interpreting ${1}/{\alpha_1''}\doteq \infty$ when $\alpha_1''=0$$)$,
consider a piecewise $C^1$ map $\phi:\R\to\Omega$ 
that
satisfies
\begin{equation}
\label{bound-indata-rich}
 \|W\circ\phi\|_{L^{\infty}(\R,\Omega^w)}\leq d,
 \qquad\qquad
\Big\|\frac{d}{dx}(W\circ\phi)\Big\|_{L^{\infty}(\R,\Omega^w)}\leq b,
\end{equation}
where $\Omega^w\doteq \{w\in\R^N\ | \ w=W(u), \ u\in\Omega\}$.
 Then, the Cauchy problem~\eqref{Cauchy-phi2}-\eqref{Cauchy-phi-in2},
 admits a classical solution $u(t,x)$ on $[0,T]\times \R$ and, for every $t\leq T$, 
 letting $w(t,x)\doteq W(u(t,x))$, there hold
\begin{equation}
\label{est1-smooth-rich}
 \|w(t,\cdot)\|_{L^{\infty}(\R,\Omega^w)}\leq  d,
 \qquad\qquad
 \|w_x(t,\cdot)\|_{L^{\infty}(\R,\Omega^w)} \leq 2 \cdot b.
\end{equation}
\end{lemma}
\begin{proof}
Proceeding as in the proof of Lemma~\ref{smooth-bounds}, 
it will be sufficient to show that, for any fixed time $T\leq 1/(2 \alpha_1'' N\cdot b)$, 
and for every initial data $\phi$ satisfying~\eqref{bound-indata-rich}, the estimates~\eqref{est1-smooth-rich} 
hold on $[0, T]$ for the Riemann coordinate expression $w(t,x)$ of
a classical solution of~\eqref{Cauchy-phi2}-\eqref{Cauchy-phi-in2}. 
Observe that the first inequality in~\eqref{est1-smooth-rich} is an immediate
consequence of the invariance of each $i$-th Riemann coordinate 
$w_i(t,x)$ along the $i$-th characteristics
of~\eqref{sys-diag}, and of the fact that $w(0,x)=W\circ \phi (x)$. Next,
defining $Q(t)\doteq \sup_i\|q_i(t,\cdot)\|_{L^\infty}$,
and relying on~\eqref{eq-qi}, \eqref{eigenval-deriv-riem}, \eqref{alpha-2w-def1}, \eqref{w-d-ball},
we derive as in~\eqref{est-Q}-\eqref{est-Q-2} the bound
\begin{equation}
\label{est-Q-4}
Q(t)  \leq \frac{Q(0)}{1-\alpha_1'' N t\,Q(0)}  \leq 2 \cdot Q(0)
\qquad\quad\forall~0\leq t < T,
\end{equation}
provided that $Q(0)\leq {1}/{(2 \alpha_1'' N\cdot T)}.$ Thus, since $Q(t)=\|w_x(t,\cdot)\|_{L^\infty}$
by  the definition of the $L^\infty$-norm, and because
$w_x(0,x)=\frac{d}{dx}(W\circ \phi)(x)$, 
if we assume $b\leq {1}/{(2 \alpha_1'' N \cdot T)}$
we recover from~\eqref{bound-indata-rich}, \eqref{est-Q-4}, 
 the second inequality in~\eqref{est1-smooth-rich}.
\end{proof}

Observe now that
as a consequence of~\eqref{normalized-2-eigenvect} we deduce also that
the rarefaction curve of the $i$-th family through $0$ can be parametrized in 
Riemann coordinates
as  $s\mapsto R^{\mathcal{R}}_i(s)\doteq s\, e_i$, $s\in \,]-\overline d', \overline d'[\,$, where $e_i$ denotes the $i$-th element
of the canonical basis of~$\R^N$.
Therefore, given $\beta\in\mathcal{PC}^1_{[d, b]}$, $d\leq \overline d'$, the map $\phi^\beta_i$ in~\eqref{j-initial wave}
takes the expression in Riemann coordinates:
\begin{equation}
\label{i-initial wave-rich}
W\circ\phi^\beta_i(x)\doteq \beta(x)\, e_i
\qquad x\in\R\,.
\end{equation}
Similarly, 
the map $\phi^\beta$ in \eqref{phidef} defined in connection with
an $N$-tuple $\beta=(\beta_1,\dots,\beta_N)\in \mathcal{PC}^{1,N}_{[L,d,b,T]}$, $d\leq \overline d'$,
is given in Riemann coordinates by 
\begin{equation}
\label{phidef-rich}
W\circ \phi^\beta(x)\doteq \sum_{i=1}^N W\circ \phi^{\beta_i}_i(x)= \sum_{i=1}^N \beta_i(x)\, e_i
=(\beta_1(x),\dots,\beta_N(x)).
\end{equation}
%
%
Notice that the supports of the simple waves $\phi_i^{\beta_i}$ may well overlap,
because we are not assuming here that $T$ satisfies
the bound~\eqref{T-assumption1}.
However, by~\eqref{phidef-rich} the structure of the solution in Riemann coordinates can be viewed as a superposition of almost decoupled simple waves 
since each $i$-th simple wave has zero $j$-th Riemann component for every $j\neq i$.
With similar arguments to the proof of Proposition~\ref{n-simple-waves-overlap} we then derive
the sharper a-priori bound on the size of the support of $w(t,\cdot)$ provided by the following
\begin{proposition}
\label{n-simple-waves-temple}
In the same setting and with the notations of Proposition~\ref{n-simple-waves-overlap}
and Lemma~\ref{smooth-bounds-rich},
assume that \eqref{Cauchy-phi-n} is a strictly hyperbolic and rich system.
Given $L, T>0$, and $d, b >0$ satisfying
\begin{equation}
\label{db-assumption3}
0<d\leq \min\bigg\{\overline d',\, \frac{\Delta_{\wedge}\lambda}{2\alpha_1''\sqrt N}\bigg\},
\qquad\quad
0<b\leq  \min\bigg\{\frac{1}{2\alpha_1'\cdot T},\, \frac{\Delta_{\wedge}\lambda}{2 \alpha_1'' N\cdot L}\bigg\},
\end{equation}
%
 with $\alpha_1', \alpha_1''$ as in~\eqref{alpha-2w-def1a}, \eqref{alpha-2w-def1},
 and $\Delta_{\wedge}\lambda$ as in~\eqref{T-assumption1}
 $($interpreting 
${1}/{\alpha_1'}\doteq \infty$
when $\alpha_1'=0$ and
${1}/{\alpha_1''}\doteq \infty$ when $\alpha_1''=0$$)$,
let $\mathcal{PC}^{1,N}_{[L,d,b,T]}$ be the class of maps introduced in~\eqref{Itauset-def2}-\eqref{PCNset},
and consider a map $\phi^\beta:\R\to \Omega$  as in~\eqref{phidef}, defined in connection with
an $N$-tuple $\beta=(\beta_1,\dots,\beta_N)\in \mathcal{PC}^{1,N}_{[L,d,b,T]}$.
 Then, the Cauchy problem~\eqref{Cauchy-phi-n}-\eqref{Cauchy-phi-in-n},
 admits a classical solution $u(t,x)$ on $[0,T]\times \R$.  Moreover,  letting $w(t,x)\doteq W(u(t,x))$, one has
\begin{equation}
\label{supp-w2}
\mathrm{Supp}(w(T,\cdot))\subseteq [-L,\, L],
\end{equation}
and, for every $t\leq T$, there hold:
  %
 \begin{equation}
 \label{est2-smooth-rich}
 \|w(t,\cdot)\|_{L^{\infty}(\R,\Omega^w)}\leq  d,
 \qquad\qquad
 \|w_x(t,\cdot)\|_{L^{\infty}(\R,\Omega^w)}\leq 4 b.
 \end{equation}
 \end{proposition}
\noindent
\begin{proof}
We shall first assume that  $T\geq {L}/{\Delta_{\wedge}\lambda}$.
In this  case, as in the proof of Proposition~\ref{n-simple-waves-overlap},
we will show that a classical solution of the Cauchy problem~\eqref{Cauchy-phi-n}-\eqref{Cauchy-phi-in-n},
satisfying~\eqref{supp-w2}, \eqref{est2-smooth-rich}, 
is obtained on $[0, T-L/\Delta_{\wedge}\lambda]$ as a superposition of simple waves
supported on disjoint set. Next, we will prove that such a solution can be extended 
to the interval $[T-L/\Delta_{\wedge}\lambda,\,T]$ relying on Lemma~\ref{smooth-bounds-rich}.
Finally, we will discuss how to derive from Lemma~\ref{smooth-bounds-rich} the existence
of a classical solution of~\eqref{Cauchy-phi-n}-\eqref{Cauchy-phi-in-n} 
verifying~\eqref{supp-w2}, \eqref{est2-smooth-rich}  in the case where
$T< {L}/{\Delta_{\wedge}\lambda}$.

\noindent{\bf 1.} 
Given $\beta=(\beta_1,\dots,\beta_N)\in \mathcal{PC}^{1,N}_{[L,d,b,T]}$,
consider the functions $x^\flat_i(t,y)$, $i=1,\dots,N$, defined in~\eqref{i-char2}.
Observe that, relying on~\eqref{eigenval-deriv-riem}-\eqref{alpha-2w-def1a},
\eqref{db-assumption3}, by the same computations of Subsection~\ref{subsec:Simple waves}
we derive the inequality
\begin{equation}
\label{char-ric-est1}
\frac{\partial}{\partial y}x^\flat_i(t,y)\geq 1-\overline\alpha'_1 \, b\cdot t
\geq \frac{1}{2}
\qquad  \forall~t\in[0,T].
\end{equation}
It follows that the maps $y\mapsto x^\flat_i(t,y)$, $i=1,\dots,N$,  are
one-to-one in $\R$, for every fixed $t\in [0,T]$. Thus, we may define
the inverse map of $x^\flat_i(\cdot ,y)$ on  $\R$, and 
setting
\begin{equation}
\label{inverse-xi-flat}
z^\flat_i(t,\cdot)\doteq\ (x^\flat_i)^{\!^{-1}}\!(t,\cdot),
\qquad i=1,\dots,N,
\end{equation}
we define the the function
\begin{equation}
\label{solution-wflat}
w_i^\flat(t,x)\doteq 
\begin{cases}
{\beta_i}(z^\flat_i(t,x))\ \ &\text{if}\quad \ x\in [x^\flat_i(t,\xi_i^-),\, x^\flat_i(t,\xi_i^+)]\setminus
\displaystyle{\bigcup_{j\neq i}}[x^\flat_j(t,\xi_j^-),\, x^\flat_j(t,\xi_j^+)],
\\
\noalign{\smallskip}
\ \ 0&\text{otherwise,}
\end{cases}
 \ \ \ i=1,\dots,N,
\end{equation}
on $[0, T] \times\R$. As in the proof of Proposition~\ref{n-simple-waves-overlap},
notice that if we assume $T\geq {L}/{\Delta_{\wedge}\lambda}$ we derive 
\begin{equation}
\label{chardist2b}
  x^\flat_{i+1}(t,\xi_{i+1}^\pm) \leq x^\flat_i(t,\xi_i^\pm)
  \qquad\forall~t\in \bigg[0, T-\frac{L}{\Delta_{\wedge}\lambda}
  \bigg],
  \qquad i=1,\dots, N-1,
\end{equation}
so that one has
\begin{equation}
\label{solution-wi2}
w_i^\flat(t,x)=
\begin{cases}
{\beta_i}(z^\flat_i(t,x))\ \ &\text{if}\quad \ x\in [x^\flat_i(t,\xi_i^-),\, x^\flat_i(t,\xi_i^+)],
\\
\noalign{\smallskip}
\ \ 0&\text{otherwise,}
\end{cases}
\ \ \ \ \ i=1,\dots,N,
\end{equation}
for all $(t,x)\in [0, T-L/\Delta_{\wedge}\lambda]\times\R$.
Relying on~\eqref{phidef-rich}, by the same arguments of the proof of Proposition~\ref{n-simple-waves-overlap}
we then deduce that $w^\flat(t,x)$ is the Riemann coordinate expression
of a classical solution of~\eqref{Cauchy-phi-n}-\eqref{Cauchy-phi-in-n} 
on $[0, T-L/\Delta_{\wedge}\lambda]\times\mathbb{R}$. Moreover, 
recalling that by definition~\eqref{PCNset} one has $\beta_i\in\mathcal{PC}^1_{[d, b]}$,
for all $i=1,\dots,N$,
and relying on~\eqref{char-ric-est1}, we deduce from~\eqref{solution-wi2} that
\begin{equation}
\label{est3swn-rich}
\|w^\flat(t,\cdot)\|_{L^{\infty}(\R,\Omega^w)} 
=\max_i\|\beta_i\|_{L^{\infty}}\leq d,
\qquad\quad
\|w^\flat_x(t,\cdot)\|_{L^{\infty}(\R,\Omega^w)} 
\leq \max_i\Big\|\Big(\frac{\partial}{\partial y}x^\flat_i(t,\cdot)\Big)^{\!\!-1}\!\cdot\!\frac{d}{dx}\beta_i\Big\|_{L^{\infty}}\leq 2b,
\end{equation}
for all $t\in [0, T-L/\Delta_{\wedge}\lambda]$.
\smallskip

\noindent
{\bf 2.}
Relying on~\eqref{db-assumption3}, \eqref{est3swn-rich}, and applying Lemma~\ref{smooth-bounds-rich}, 
we deduce the existence of a classical solution of~\eqref{Cauchy-phi-n}
on the domain $[T-L/\Delta_{\wedge}\lambda,\, T]\times\R$, that assumes the 
initial data 
\begin{equation}
\phi(x)\doteq W^{-1}(w^\flat(T-L/\Delta_{\wedge}\lambda,\,x)), \qquad x\in\R,
\end{equation}
at time $t=T-L/\Delta_{\wedge}\lambda$.
Moreover, the Riemann coordinate expression $w^\sharp(t,x)$ 
of such a solution satisfies the estimates
\begin{equation}
 \label{est3-smooth-rich}
 \|w^\sharp(t,\cdot)\|_{L^{\infty}(\R,\Omega^w)}\leq  d,
 \qquad\qquad
 \|w^\sharp_x(t,\cdot)\|_{L^{\infty}(\R,\Omega^w)}\leq 4 \cdot b,
 \end{equation}
for all $t\in [T-L/\Delta_{\wedge}\lambda,\, T]$.
Therefore, the function defined by
\begin{equation}
\label{w-ufl-ush}
w(t,x)\doteq
\begin{cases}
w^\flat(t,x)\quad &\text{if}\qquad t\in [0,\, T-L/\Delta_{\wedge}\lambda],
\\
\noalign{\smallskip}
w^\sharp(t,x)\quad &\text{if}\qquad t\in \,]T-L/\Delta_{\wedge}\lambda, \, T],
\end{cases}
\end{equation}
provides the Riemann coordinate expression of a classical solution of~\eqref{Cauchy-phi-n}-\eqref{Cauchy-phi-in-n} 
on $[0,T]\times\R$ that, 
because of~\eqref{est3swn-rich}, \eqref{est3-smooth-rich},  
satisfies the bounds~\eqref{est2-smooth-rich} for all $t\in [0,T]$. 

In order to derive an estimate on the support of $w(T,\cdot)$, consider the $i$-th characteristic $t\mapsto x_i(t,y)$ starting at $y$
at time $t=0$, associated to $w(t,x)$, 
i.e. the solution of~\eqref{char-w-def}.
Since $w(t,x)$ is the Riemann coordinate expression of 
a classical solution of~\eqref{Cauchy-phi-n} on $[0,T]\times\R$, it follows that 
the map $y\mapsto x_i(t,y)$ is a one-to-one correspondence on~$\R$, for any $t\in [0,T]$.
Hence, setting $z_i(t,\cdot)\doteq x_i^{-1}(t,\cdot)$,
and recalling that  each $i$-th Riemann coordinate 
$w_i(t,x)$ remains constant along the $i$-th characteristics, 
we may express $w_i(t,x)$ as
\begin{equation}
\label{solution-wi}
w_i(t,x)=\beta_i(z_i(t,x))\qquad \forall~t\in[0,T],\ \ x\in\R.
\end{equation}
Relying on~\eqref{solution-wi}, and because of~\eqref{PCNset},
we deduce that in order to prove~\eqref{supp-w2}
it will be sufficient to show that the $i$-th characteristic map $x_i(T,\cdot)$ satisfies 
\begin{equation}
\label{supp-w3}
[x_i(T,\xi_i^-),\, x_i(T,\xi^+)]\subseteq [-L,\, L],
\end{equation}
for every $i=1,\dots,N$. To this end, let $t\mapsto x_i(t;\tau,y)$ denote the $i$-th characteristic
starting at $y$ at time $t=\tau$, i.e. the solution of 
\begin{equation}
\label{char-w-def2}
\dot x = \lambda_i(w(t,x)),
\qquad\quad x(\tau)=y,
\end{equation}
and define the times $\tau_i^\pm$ and points $y_i^\pm$ as in~\eqref{tau-i-def}.
Then, recalling~\eqref{i-char2}, thanks to~\eqref{Itauset-def2}, \eqref{char-i-T}, \eqref{bound-taui},
\eqref{uchar-i-T},
\eqref{bound-grad-lambda-w}, \eqref{est2-smooth-rich},
\eqref{char-w-def2}, and because of~\eqref{db-assumption3}, we find
\begin{equation}
\label{supp-est5}
\begin{aligned}
x_i(T, \xi_i^\pm)
&\leq y_i^\pm + \bigg(\lambda_i(0)+\sqrt N\,\overline\alpha''_1 \cdot d
\bigg)\cdot (T-\tau_i^\pm)
\\
\noalign{\smallskip}
&= x^\flat_i(T,\,\xi_i^\pm) +\sqrt N\,\overline\alpha''_1 \cdot d
\cdot  (T-\tau_i^\pm)
\\
\noalign{\smallskip}
&\leq \frac{L}{2}+\sqrt N\,\overline\alpha''_1 \cdot d\cdot  \frac{L}{\Delta_{\wedge}\lambda}<L\,.
\end{aligned}
\end{equation}
With similar arguments we derive $x_i(T, \xi_i^\pm)>-L$, which together with~\eqref{supp-est5}, yields~\eqref{supp-w3}.
This completes the proof of the proposition in the case where $T\geq {L}/{\Delta_{\wedge}\lambda}$.
\smallskip

\noindent
{\bf 3.}
Assume $T< {L}/{\Delta_{\wedge}\lambda}$, and observe that by~\eqref{db-assumption3}
one has $b\leq 1/{(2 \alpha_1'' N \cdot T)}$.
Then, applying Lemma~\ref{smooth-bounds-rich}, 
we deduce the existence of a classical solution of~\eqref{Cauchy-phi-n}-\eqref{Cauchy-phi-in-n} 
on $[0,T]\times\R$ that
satisfies the bounds~\eqref{est2-smooth-rich} for all $t\in [0,T]$. Letting $w(t,x)$
denote the Riemann coordinate expression of such a solution, by the same arguments above
we can show that~\eqref{solution-wi}, \eqref{supp-w3} hold, 
which, together with~\eqref{PCNset}, yield~\eqref{supp-w2}, thus concluding the proof of the proposition.
\end{proof}
\begin{remark} 
\label{uniqueness-temple}
The same conclusion of Remark~\ref{uniqueness} holds if we consider a 
rich system
that generates a
semigroup of entropy weak solutions $(S_t)_{t\geq 0}$ with a domain $\mathcal{D}$ as in~\eqref{domtemple},
and we assume that 
%
\begin{equation}
\label{domain-temple}
]\overline d', \overline d'[\,\subset  [a_i, b_i] \qquad\quad\forall~i=1,\dots, N.
\end{equation}
In fact, under this assumption
it clearly follows 
that, for every given $\beta=(\beta_1,\dots,\beta_N)\in \mathcal{PC}^{1,N}_{[L,d,b,T]}$,
$d< \overline d'$,
 the map $\phi^\beta$ defined in~\eqref{phidef} belongs to $\mathcal{D}$,
and thus, relying on~\cite[Section 5.3]{Dafermos:Book}, we deduce that
the classical solution $u(t,\cdot)$ of the Cauchy problem~\eqref{Cauchy-phi-n}-\eqref{Cauchy-phi-in-n}
provided by Proposition~\ref{n-simple-waves-temple}
coincides with 
$S_t \phi^\beta$.
\end{remark}

\section{Lower compactness estimates for conservation laws}
\label{sec:lower-est}

\subsection{A controllability result}
\label{subsec-control}

For arbitrary constants $L,  b>0$,  $0<M< \overline d$ \,
(\,$\overline d$ being the 
radius of the ball contained in the domain of the flux function where condition~\eqref{lambdadef} is verified),
and $T>0$ satisfying~\eqref{T-assumption1},
recalling the definitions~\eqref{Itauset-def2}, \eqref{PCNset}, \eqref{phidef},
let us consider the set
\begin{equation}
\label{attainable set-T}
\mathcal{A}_{[L,M,b, T]}
\doteq
\Big\{\psi\in C(\R,\Omega)\,\big|\,
\psi(x)=\phi^\beta(-x)\ \ \forall~x\in\R,\ \ \ \text{for some}\ \ \ \beta=(\beta_1,\dots,\beta_N)\in \mathcal{PC}^{1,N}_{[L,M,b,T]} \
\Big\}.
\end{equation}
Notice that, because of~\eqref{Itauset-def2}, 
every map $\psi\in \mathcal{A}_{[L,m,M,b, T]}$ is supported on $N$ disjoint intervals $[\xi_i^-, \xi_i^+]$,
$i=1,\dots, N$,
of length $L$.
The next result shows that the elements of such a set can be obtained as the values
$S_T\overline u$ at a fixed time~T of the semigroup generated by~\eqref{conlaws},
for initial data $\overline u$ varying in a set of the form~\eqref{Cclass}.
\begin{proposition}
\label{controllability}
Let $f:\Omega \to \R^N$ be a $C^2$ map on an open, connected domain $\Omega\subset\R^N$ containing the origin,
and assume that the system~\eqref{conlaws} is strictly hyperbolic.
Let $(S_t)_{\geq 0}$ be the semigroup of entropy weak solutions generated by~\eqref{conlaws} 
defined on a domain $\mathcal{D}_0$   satisfying~\eqref{domainsgr}.
Then, 
given any $L, m, M, T>0$,
 and setting 
%
\begin{equation}
\label{IL-Ltilde-def}
\widetilde L\doteq \min\left\{\frac{L}{(1+\alpha_5)},\ T\cdot\Delta_{\wedge}\lambda\right\},
\end{equation}
$($$\Delta_{\wedge}\lambda,\, \alpha_5$ being the constants in~\eqref{T-assumption1}, \eqref{alfa4def}$)$, 
for every
\begin{equation}
\label{bh-bound}
\begin{gathered}
0\leq b\leq  \min\bigg\lbrace{
\frac{1}{2\alpha_1\cdot T},\, 
\frac{\Delta_{\wedge}\lambda}{4 \alpha_3\alpha_4 N^2\cdot L},\, 
\frac{\delta_0}{8\alpha_4 N L}\bigg\rbrace},
\\
\noalign{\medskip}
0\leq h \leq  
\min\bigg\lbrace{\frac{\overline d}{2 \alpha_4 N\, e^{\alpha_2/\alpha_3}}, \, 
\frac{\Delta_{\wedge}\lambda}{4\alpha_1\,\alpha_4 N\, e^{\alpha_2/\alpha_3}}, \, 
\frac{M}{2\alpha_4 N\, e^{\alpha_2/\alpha_3} },\, \frac{m}{2L}
\bigg\rbrace}
\end{gathered}
\end{equation}
$($$\alpha_l$, $l=1,2,3,4$, being the constants in~\eqref{lambdadef}, \eqref{alpha-1-2-def2},
interpreting $1/\alpha_1\doteq \infty$
when $\alpha_1=0$
$)$, there holds
\begin{equation}
\label{att-incl}
\mathcal{A}_{[\widetilde L,h,b,T]}\subset
S_T\big(\mathcal{L}_{[I_L,m,M]}\cap{\mathcal{D}_0}\big)\,,
\qquad\quad I_L\doteq [-L, L],
\end{equation}
where $\mathcal{A}_{[\widetilde L,h,b,T]}$, $\mathcal{L}_{[I_L,m,M]}$
denote the sets defined as in~\eqref{attainable set-T}, \eqref{Cclass}, respectively.
%
\end{proposition}
\begin{proof}
Following the same strategy adopted in~\cite{AOK}, 
we will show that any element $\psi\in\mathcal{A}_{[\widetilde L,h,b,T]}$ 
can be obtained as the value at time $T$ of a 
classical solution to~\eqref{conlaws} by reversing the direction of time,
and constructing a backward solution to~\eqref{conlaws} 
that starts at time $T$ from $\psi$.
Namely, given
\begin{equation}
\label{final-data}
\psi\in \mathcal{A}_{[\widetilde L,h,b,T]}\,,
\end{equation}
by definition~\eqref{attainable set-T} there will be an $N$-tuple of maps 
$\beta=(\beta_1,\dots,\beta_N)\in \mathcal{PC}^{1,N}_{[\widetilde L,h,b,T]}$,
such that letting $\phi^\beta$ be the function defined in~\eqref{phidef}, one has $\psi(x)=\phi^\beta(-x)$,
for all $x$. Notice that, by~\eqref{IL-Ltilde-def}, one has
\begin{equation}
\label{T-bound}
T\geq\frac{\widetilde L}{\Delta_{\wedge}\lambda},
\end{equation}
as in~\eqref{T-assumption1},
while~\eqref{bh-bound} imply that $h, b$ satisfy the bounds~\eqref{db-assumption2} on $d,b$. 
Then, set 
\begin{equation}
\label{om-indata}
\omega_0(x)\doteq \psi(-x)=\phi^\beta(x)\qquad\forall~x\in\R,
\end{equation}
and let $\omega : [0,T]\times \R \to \Omega$ denote the classical solution
of the Cauchy problem~\eqref{Cauchy-phi-n}-\eqref{Cauchy-phi-in-n}, provided by Proposition~\ref{n-simple-waves-overlap}.
%
Next, consider the function
\begin{equation}\label{EqDefu}
u(t,x)\doteq \omega(T-t,-x), \qquad  (t,x)\in [0,T] \times \R.
\end{equation}
Observe that $u(t,x)$ is a classical solution of~\eqref{conlaws} with initial data $u(0,\cdot)=\omega(T,-\cdot)$
that, by~\eqref{om-indata}, satisfies
\begin{equation}
\label{u-psi}
u(T,\cdot)=\psi.
\end{equation}
Moreover, by~\eqref{supp-w}, \eqref{IL-Ltilde-def} 
we have $|\text{Supp}(\omega(T,-\cdot))|=|\text{Supp}(\omega(T,\cdot))|\leq 2(1+\alpha_5)\widetilde L\leq 2L$.
Therefore, relying on the second estimate in~\eqref{est1swn} and on~\eqref{bh-bound},
we derive
%
\begin{equation}
\begin{aligned}
\mathrm{Tot. var. }(\omega(T,-\cdot))
&\leq
\|\omega_x(T,\cdot)\|_{L^1(\mathbb{R},\Omega)}
%
\\
\noalign{\smallskip}
&\leq \|\omega_x(T,\cdot)\|_{L^\infty(\mathbb{R},\Omega)}\cdot |\text{Supp}(\omega(T,\cdot))|
\\
\noalign{\smallskip}
&\leq 4\alpha_4 N\cdot b\cdot 2L
\leq \delta_0.
\end{aligned}
\end{equation}
Thus, by~\eqref{domainsgr} we deduce 
that $u(0,\cdot)=\omega(T,-\cdot)\in \mathcal{D}_0$,
and hence, recalling 
Remark~\ref{uniqueness}, we have $u(t,\cdot)=S_t \omega(T,-\cdot)$,
for all $t\in [0,T]$. Because of~\eqref{u-psi}, this implies in particular that $\psi=S_T \omega(T,-\cdot)$.
To conclude the proof of
\begin{equation}
\label{psi-in-attset}
\psi\in S_T\big(\mathcal{L}_{[I_L,m,M]}\cap{\mathcal{D}_0}\big)
\end{equation}
it thus remains to show that
\begin{equation}
\label{psi-in-Lset}
\omega(T,-\cdot)\in \mathcal{L}_{[I_L,m,M]}.
\end{equation}
Since $\omega$ is the classical solution
of~\eqref{Cauchy-phi-n}-\eqref{Cauchy-phi-in-n} provided by Proposition~\ref{n-simple-waves-overlap},
recalling that $\psi(-\cdot)=\phi^\beta$, 
$\beta\in \mathcal{PC}^{1,N}_{[\widetilde L,h,b,T]}$,
and 
relying on~\eqref{supp-w}, \eqref{est1swn}, \eqref{IL-Ltilde-def},
\eqref{bh-bound}, \eqref{final-data},
we deduce that
\begin{equation}
\label{supp-norm-omT}
\begin{gathered}
\mathrm{Supp}(\omega(T,-\cdot))\subset [-L, L],
\\
\noalign{\smallskip}
\|\omega(T,-\cdot)\|_{L^\infty(\mathbb{R},\Omega)}\leq
2\alpha_4 N\, e^{\frac{\alpha_2}{\alpha_3}}\cdot h \leq M,
\\
\noalign{\smallskip}
\|\omega(T,-\cdot)\|_{L^1(\mathbb{R},\Omega)} \leq 
2L h\leq m.
\end{gathered}
\end{equation}
Therefore, the inclusion~\eqref{psi-in-Lset} is verified because of~\eqref{supp-norm-omT}, which completes
the proof of the proposition.
\end{proof}
%
%
\begin{remark} 
\label{rem-scalar-control}
When $N=1$, under the same assumptions as Proposition~\ref{controllability}, assume also that $f'(0)=0$
(possibly performing a space and flux transformation).
Then, relying on Lemma~\ref{i-simple-wave} $($where we may reach the same conclusion assuming that $b\leq 3/(4 \overline c\cdot T)$,
with $\overline c\doteq \sup\big\{|f''(u)|\,; \ |u|\leq \overline d\big\}$$)$, we can show that the following holds.
Given any $L, m, M, T>0$, for every
\begin{equation*}
0\leq b\leq \frac{3}{4 \overline c\cdot T},\qquad\quad
0\leq h \leq 
\min\bigg\lbrace{\overline d',\, M,\, \frac{m}{2L}
\bigg\rbrace},
\end{equation*}
one has
\begin{equation}
\label{att-incl-scakar}
\mathcal{A}_{[L,h,b,T]}\subset
S_T\big(\mathcal{L}_{[I_L,m,M]}\big)\,,
\qquad\quad I_L\doteq [-L/2, L/2],
\end{equation}
where $\mathcal{A}_{[L,h,b,T]}$, $\mathcal{L}_{[I_L,m,M]}$
denote sets defined as in~\eqref{attainable set-T}, \eqref{Cclass}, respectively.

\end{remark}
%

We shall now extend the previous controllability results to class 
of functions with possibly unbounded total variation in the case of hyperbolic
systems of conservation laws of Temple class.
We recall that  (see~\cite{Dafermos:Book,serre,temple}):
\begin{definition}
\label{def-temple}
A system of conservation laws~\eqref{conlaws} is called of {\it Temple class}
if:
\begin{itemize}
\item[-] it is a rich system, i.e. 
if it is endowed with a coordinates system $w=(w_1,\ldots,w_n)$  of Riemann invariants $w_i=W_i(u)$
associated to each characteristic field $r_i$;
\item[-] the level sets \textrm{$\big\{u\in \Omega;~W_i(u)=constant\big\}$}
of every Riemann invariant
are hyperplanes. 
\end{itemize}
\end{definition}
\noindent
We shall 
assume that $W(0)=0$ and that as $w$ ranges within the product set 
$\Pi \doteq [a_1, b_1]\times\cdots\times [a_N,b_N]$, the corresponding state $u=W^{-1}(w)$ 
remains inside the domain $\Omega$ of the flux function $f$.

We also recall that a characteristic field $r_i$ of a system~\eqref{conlaws} is said to be {\it genuinely nonlinear} (GNL) in the sense of Lax
if $\nabla \lambda_i(u)\cdot r_i(u)\neq 0$ for all $u\in\Omega$, while we say that 
$r_i$ is {\it linearly degenerate} (LD) if $\nabla \lambda_i(u)\cdot r_i(u)\equiv 0$ for all $u\in\Omega$.
\medskip

As observed in the introduction, the results in~\cite{bianc}, \cite{brgoa1} show
that a Temple system with GNL or LD characteristic families admits a continuous
semigroup of entropy weak solutions $S:  [0,\infty[ \times \mathcal{D}\to \mathcal{D}\}$
defined on domains $\mathcal{D}$ as in~\eqref{domtemple} of functions having
possibly unbounded variation.
We shall adopt the notation
$S^w_t \overline w\doteq W(u(t,\cdot))$ for
the Riemann coordinates expression of the
solution of~\eqref{conlaws},\eqref{indata}, with $\overline u\doteq W^{-1}\circ\overline w$.
Therefore, relying on the sharper a-priori bounds on the classical solutions of a rich system
provided by Proposition~\ref{n-simple-waves-temple}, 
and setting
\begin{equation}
\label{attainable set-T-w}
\mathcal{A}^w_{[L,M,b, T]}
\doteq
\Big\{\psi\in C(\R,\Pi)\,\big|\,
\psi(x)=\beta(-x)\ \ \forall~x\in\R,\ \ \ \text{for some}\ \ \ \beta=(\beta_1,\dots,\beta_N)\in \mathcal{PC}^{1,N}_{[L,M,b,T]} \
\Big\}.
\end{equation}
we establish the following
\begin{proposition}
\label{controllability-temple}
In the same setting of Proposition~\ref{controllability},
assume that~\eqref{conlaws} is a strictly hyperbolic
system of Temple class, 
and that all characteristic families are genuinely nonlinear
or linearly degenerate. 
Let $(S_t)_{\geq 0}$ be the semigroup of entropy weak solutions generated by~\eqref{conlaws} 
defined on a domain $\mathcal{D}$  as in~\eqref{domtemple},
and assume that \eqref{domain-temple} holds.
Then,  given any $L, m, M, T>0$,
for every $b, h$ satisfying
\begin{equation}
\label{bh-bound-temple}
\begin{gathered}
0\leq b\leq \min \bigg\{\frac{1}{2\alpha_1'\cdot T},\, \frac{\Delta_{\wedge}\lambda}{2 \alpha_1'' N\cdot L}\bigg\},
\qquad\quad
0\leq h \leq 
\min\bigg\lbrace{\overline d',\, \frac{\Delta_{\wedge}\lambda}{2\alpha_1''\sqrt N}, \, 
M,\, \frac{m}{2L}
\bigg\rbrace},
\end{gathered}
\end{equation}
$($$\alpha_1', \alpha_1''$  being the constants in~\eqref{alpha-2w-def1a}, \eqref{alpha-2w-def1}
 and $\Delta_{\wedge}\lambda$ as in~\eqref{T-assumption1},
interpreting $1/\alpha_1'\doteq \infty$
when $\alpha_1'=0$ and $1/\alpha_1'' \doteq \infty$ 
when $\alpha_1''=0$$)$ there holds
\begin{equation}
\label{att-incl-temple}
\mathcal{A}^w_{[L,h,b,T]}\subset
S^w_T\big(\mathcal{L}^w_{[I_L,m,M]}\big)\,,
\qquad\quad 
I_L\doteq [-L, L],
\end{equation}
where the sets $\mathcal{A}^w_{[L,h,b,T]}$, $\mathcal{L}^w_{[I_L,m,M]}$, are defined as in~\eqref{attainable set-T-w}
and in~\eqref{Cclass-w}, respectively.
%
%
\end{proposition}
\begin{proof}
The proof of Proposition~\ref{controllability-temple} is entirely similar  to that of Proposition~\ref{controllability},
relying on Proposition~\ref{n-simple-waves-temple} and Remark~\ref{uniqueness-temple}, 
and recalling~\eqref{phidef-rich}, thus we omit it.
\end{proof}

\subsection{Lower compactness estimates on a family of simple waves}

We shall provide now a lower estimate on the $\varepsilon$-entropy
of the class $\mathcal{A}_{[L,m,M,b]}$ introduced in~\eqref{attainable set-T}.
To this end, set
\begin{equation}
\label{a5-def}
\alpha_6\doteq
\sup\Big\{\big|D r_i(u)\big|\,; \, u\in B_{\overline d},\,  i =1,\dots,N\Big\},
\end{equation}
where $ B_{\overline d}$ denotes as usual a ball centered in the origin and contained in the domain $\Omega$ of the flux
function~$f$.
Following a similar strategy as the one pursued in~\cite{AOK} we then establish the following
\begin{proposition}
\label{Combinatory}
In the same setting of Proposition~\ref{controllability}, given $L, b>0$, $0<M< \overline d$, and $T>0$ satisfying~\eqref{T-assumption1},
for every $\varepsilon >0$ satisfying 
\begin{equation}
\label{cond1}
\varepsilon\leq\min\bigg\{\frac{LNM}{24},\, \frac{LN}{48\alpha_6}\bigg\},
\end{equation}
$($$\alpha_6$ being the constant in~\eqref{a5-def}$)$,
letting $\mathcal{A}_{[L,M,b, T]}$ be the set defined in~\eqref{attainable set-T},
one has
\begin{equation}
\label{low-bound-entr-A}
H_{\varepsilon}\Big(\mathcal{A}_{[L,M,b, T]}\ |\ L^1(\mathbb{R},\Omega)\Big)\geq \frac{L^2N^2 b}{216\, \ln (2)}\cdot\frac{1}{\varepsilon}.
\end{equation}
\end{proposition}
\begin{proof}
Towards a proof of~\eqref{low-bound-entr-A}, we shall first introduce a two-parameter family $\mathcal{F}_{n,h}\subset
\mathcal{A}_{[L,M,b, T]}$, depending on $n \geq 2$ and $h>0$,
of superposition of simple waves $\phi^\beta$, 
$\beta=(\beta_1,\dots,\beta_N)
\subset  \mathcal{PC}^{1,N}_{[L,M,b,T]}$,
defined as in~\eqref{phidef} in connection with piecewise affine, compactly supported maps 
$\beta_i\in \mathcal{PC}^1_{[M,b]}$. Next, we shall provide an optimal lower bound on the covering number 
$N_\varepsilon(\mathcal{F}_{n,h}\,| \, L^1(\mathbb{R},\Omega))$, for a suitable choice of $n,h$, which, in turn, 
will yield the lower bound~\eqref{low-bound-entr-A} on the $\varepsilon$-entropy of $\mathcal{A}_{[L,M,b, T]}$.
\medskip

\noindent
{\bf 1.} 
Given any integer $n\geq 2$ and any constant $h>0$, 
for every $k$-th characteristic family
and for any given $n$-tuple $\iota=(\iota_1,\dots,\iota_n)\in\lbrace{0,1\rbrace}^n$, 
we consider 
the function $\beta_k^{\iota}:\R\rightarrow [-h,h]$, 
with support contained in $[\xi_k^-, \xi_k^+]$, defined by
setting
(see Figure \ref{pic}) 
\begin{equation}
\label{DefFiota}
\beta_k^{\iota}(x) \doteq
(-1)^{\iota_\ell}
\,\frac{2hn}{L}\!\cdot\! \bigg( \frac{L}{2n}-\Big|x-\xi_k^- -(2\ell+1)\cdot\frac{L}{2n}\Big| \bigg)
\qquad\quad
\forall~x\in\!\bigg[\xi_k^- \!+\!\frac{\ell \,L}{n N},\ \xi_k^- \!+\!\frac{(\ell+1) L}{n N}\bigg],
\end{equation}
for all $\ell \in \{0,\dots,n-1\}$. Recall that the quantities  $\xi_k^\pm=\pm L/2-\lambda_k(0)\cdot T$ were introduced in~\eqref{Itauset-def2}.
\begin{center} 
\begin{figure}[htbp]
\begin{center}
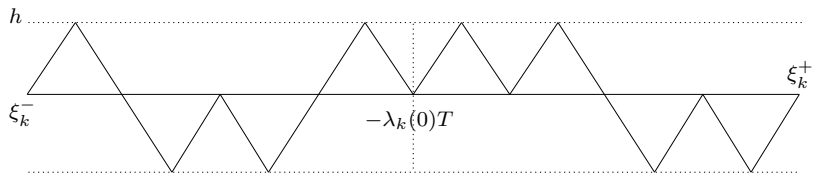
\end{center}
\caption{The function $\beta_k^{\iota}$ for $n=8$ and $\iota=(0,1,1,0,0,0,1,1)$.}
\label{pic}
\end{figure}
\end{center}
%
%
Observe that, if we assume
\begin{equation}
\label{assumption-h}
0<h\leq \min\bigg\lbrace{M,\ \frac{Lb}{2n}
\bigg\rbrace},
\end{equation}
by definition~\eqref{PCset}
it follows that  $\beta_k^{\iota}\in \mathcal{PC}^1_{[M,b]}$, for every  $n$-tuple $\iota=(\iota_1,\dots,\iota_n)\in\lbrace{0,1\rbrace}^n$.
Therefore, for any given $N$-tuple of $n$-tuples $(\iota_1,\dots,\iota_N)\in(\lbrace{0,1\rbrace}^n)^N$,
letting $\beta_k^{\iota_k}$, $k=1,\dots,N$, be the maps defined as in~\eqref{DefFiota}, 
and recalling definition~\eqref{PCNset}, we deduce that 
$(\beta_1^{\iota_1},\dots,\beta_N^{\iota_N})\in \mathcal{PC}^{1,N}_{[L,M,b,T]}$. 
Thus, 
for all $n\geq 2$ and $h$ satisfying \eqref{assumption-h},
setting
\begin{equation}
\label{Bnhset-def}
\begin{aligned}
\mathcal{B}_{n,h}\doteq
\Big\{
(\beta_1^{\iota_1},\dots,\beta_N^{\iota_N})
\ \big| \
\beta_k^{\iota_k}:\R\to [-h,h]\ \ \text{defined as in~\eqref{DefFiota}
with}& \ \mathrm{Supp}(\beta_k^{\iota_k})\subset [\xi_k^-,\xi_k^+]\ \  \forall~k, 
\\
\noalign{\smallskip}
&\qquad(\iota_1,\dots,\iota_N)\in(\lbrace{0,1\rbrace}^n)^N
\Big\},
\end{aligned}
\end{equation}
one has
\begin{equation}
\label{Bnhsubet}
\mathcal{B}_{n,h}\subset \mathcal{PC}^{1,N}_{[L,M,b,T]}.
\end{equation}
Then, for any given $N$-tuple of $n$-tuples $(\iota_1,\dots,\iota_N)\in(\lbrace{0,1\rbrace}^n)^N$,
let
\begin{equation}
\nonumber
\phi^{\iota_1,\dots,\iota_N}\doteq
\phi^{(\beta_1^{\iota_1},\dots,\beta_N^{\iota_N})}
\end{equation}
denote  the map defined as in~\eqref{phidef} in connection
with the $N$-tuple $(\beta_1^{\iota_1},\dots,\beta_N^{\iota_N})\in \mathcal{B}_{n,h}$,
and set
\begin{equation}
\label{Fnhset-def}
\mathcal{F}_{n,h}\doteq
\Big\{
\phi^{\iota_1,\dots,\iota_N}(-\cdot)
\ \big| \
(\iota_1,\dots,\iota_N)\in(\lbrace{0,1\rbrace}^n)^N
\Big\}.
\end{equation}
Recalling definition~\eqref{attainable set-T}, and because of~\eqref{Bnhsubet},
it follows that there holds
\begin{equation}
\label{Fnhsubset}
\mathcal{F}_{n,h}\subset  \mathcal{A}_{[L,M,b, T]},
\end{equation}
for all $n\geq 2$ and $h>0$ satisfying~\eqref{assumption-h}.
Therefore, observing that~\eqref{Fnhset-def} implies
\begin{equation}
\label{ent-ineq}
H_{\varepsilon}\Big(\mathcal{A}_{[L,M,b, T]}\ |\ L^1(\mathbb{R},\Omega)\Big)\geq
H_{\varepsilon}\Big(\mathcal{F}_{n,h}\ |\ L^1(\mathbb{R},\Omega)\Big),
\end{equation}
we deduce that, in order to establish~\eqref{low-bound-entr-A},
it will be sufficient to show 
\begin{equation}
\label{low-bound-entr-Fnh}
H_{\varepsilon}\Big(\mathcal{F}_{n,h}\ |\ L^1(\mathbb{R},\Omega)\Big)\geq \frac{L^2N^2 b}{216\, \ln (2)}\cdot\frac{1}{\varepsilon},
\end{equation}
%
for a suitable choice of $n\geq 2$ and $h>0$ satisfying~\eqref{assumption-h}.
\medskip

\noindent
{\bf 2.}
Towards an estimate of the covering number $N_{\varepsilon}(\mathcal{F}_{n,h}\ |\ L^{1}(\R,\Omega))$,
let us denote with $\mathcal{C}^{\mathcal{F}}_n(\varepsilon)$ the maximum
number of elements in $\mathcal{F}_{n,h}$ that have $L^1$-distance less than $\varepsilon$
from a given element of $\mathcal{F}_{n,h}$. Namely, for any fixed 
$\phi^{\bar\iota_1,\dots,\bar\iota_N}(-\cdot)\doteq
\phi^{(\beta_1^{\bar\iota_1},\dots,\beta_N^{\bar\iota_N})}(-\cdot)\in \mathcal{F}_{n,h}$,
$\bar\iota\doteq(\bar\iota_1,\dots,\bar\iota_N)\in(\lbrace{0,1\rbrace}^n)^N$, 
define
\begin{equation}
\label{CeF-def}
\mathcal{C}^{\mathcal{F}}_{n,\bar\iota}(\varepsilon)
\doteq
\mbox{Card} \Big\{
\phi^{\iota_1,\dots,\iota_N}(-\cdot)
\in\mathcal{F}_{n,h}
\ \big| \
\big\|
\phi^{\iota_1,\dots,\iota_N}-
\phi^{\bar\iota_1,\dots,\bar\iota_N}
\big\|_{L^{1}(\R,\Omega)}\leq \varepsilon
\Big\},
\end{equation}
%
and set
\begin{equation}
\label{CeF-def2}
\mathcal{C}^{\mathcal{F}}_n(\varepsilon)
\doteq
\max\Big\{ \mathcal{C}^{\mathcal{F}}_{n,\iota}(\varepsilon)
\ \big| \
\iota\doteq(\iota_1,\dots,\iota_N)\in(\lbrace{0,1\rbrace}^n)^N
\Big\}.
\end{equation}
Observe that any element of an $\varepsilon$-cover of $\mathcal{F}_{n,h}$
contains at most $\mathcal{C}^{\mathcal{F}}_n(2\varepsilon)$
functions of $\mathcal{F}_{n,h}$. Thus, since the cardinality of $\mathcal{F}_{n,h}$
is the same of the set $\mathcal{B}_{n,h}$, which is $2^{nN}$, it follows that 
the number of sets in an $\varepsilon$-cover of $\mathcal{F}_{n,h}$ is at least
\begin{equation}
\label{CeF-est1}
N_{\varepsilon}(\mathcal{F}_{n,h}\ |\ L^{1}(\R,\Omega))\geq
\frac{2^{nN}}{\mathcal{C}^{\mathcal{F}}_n(2\varepsilon)}.
\end{equation}
Therefore, we wish to provide now an upper bound on $\mathcal{C}^{\mathcal{F}}_n(2\varepsilon)$.
To this end, consider any two $N$-tuples $\bar{\iota}\neq\iota\in(\lbrace{0,1\rbrace}^n)^N$,
$\bar\iota\doteq(\bar\iota_1,\dots,\bar\iota_N), \iota\doteq(\iota_1,\dots,\iota_N)$,
and let $\phi^{\bar\iota_1,\dots,\bar\iota_N}\doteq
\phi^{(\beta_1^{\bar\iota_1},\dots,\beta_N^{\bar\iota_N})}$,
$\phi^{\iota_1,\dots,\iota_N}\doteq
\phi^{(\beta_1^{\iota_1},\dots,\beta_N^{\iota_N})}$,
denote the maps defined as in~\eqref{phidef} in connection with the
corresponding $N$-tuples $\beta^{\bar\iota_1,\dots,\bar\iota_N}\doteq
(\beta_1^{\bar\iota_1},\dots,\beta_N^{\bar\iota_N})$ and
$\beta^{\iota_1,\dots,\iota_N}\doteq (\beta_1^{\iota_1},\dots,\beta_N^{\iota_N})$
of $\mathcal{B}_{n,h}$. 
Recall that the eigenvectors are normalized so that $|r_k(u)|\equiv 1$,
for all $k=1,\dots,N$. Moreover, by definitions~\eqref{phidef}, \eqref{Bnhset-def},
and because of~\eqref{T-assumption1},
the maps $\phi^{\bar\iota_1,\dots,\bar\iota_N}$,
$\phi^{\iota_1,\dots,\iota_N}$ 
and $\beta^{\bar\iota_1,\dots,\bar\iota_N}$,
$\beta^{\iota_1,\dots,\iota_N}$ 
are supported on the disjoint union of  sets $[\xi_k^-,\xi_k^+]$,
$k=1,\dots,N$.
Thus, recalling~\eqref{a5-def}, we find
\begin{equation}
\label{L1-phi-est}
\begin{aligned}
\big\|\phi^{\bar\iota_1,\dots,\bar\iota_N}-\phi^{\iota_1,\dots,\iota_N}\big\|_{L^1(\R,\Omega)}
&=\sum_{k=1}^N \int_{[\xi_k^-,\xi_k^+]}\Big|R_k(\beta_k^{\bar\iota_k}(x))-R_k(\beta_k^{\iota_k}(x))\Big|dx
\\
\noalign{\smallskip}
&=\sum_{k=1}^N\int_{[\xi_k^-,\xi_k^+]}\bigg|\int_{\beta_k^{\iota_k}(x)}^{\beta_k^{\bar\iota_k}(x)}r_k(R_k(s))ds\bigg|dx 
\\
\noalign{\smallskip}
&\geq\sum_{k=1}^N\Bigg[
\int_{[\xi_k^-,\xi_k^+]}\bigg[\big|\beta_k^{\bar\iota_k}(x)-\beta_k^{\iota}(x)\big|-
\bigg|
\int_{\beta_k^{\iota_k}(x)}^{\beta_k^{\bar\iota_k}(x)}\big|r_k(R_k(s))-r_k(0)\big|ds
\bigg|
\,\bigg]dx\Bigg]
\\
\noalign{\smallskip}
&\geq\sum_{k=1}^N\int_{[\xi_k^-,\xi_k^+]}\big|\beta_k^{\bar\iota_k}(x)-\beta_k^{\iota}(x)\big|\cdot\Big[1-\|D r_k\|_{L^{\infty}(B_h, M_n(\R))}\cdot h\Big]dx
\\
\noalign{\smallskip}
&\geq
\sum_{k=1}^N \big\| \beta_k^{\bar\iota_k}-\beta_k^{\iota}\big\|_{L^1(\R,\R)}
\cdot\Big[1-\alpha_6\cdot h\Big],
\end{aligned}
\end{equation}
where $\alpha_6$ is the constant in~\eqref{a5-def}.
Hence, if we assume that
\begin{equation}
\label{assumption-h2}
0<h\leq\frac{1}{2\alpha_6},
\end{equation}
it follows from~\eqref{L1-phi-est} that, adopting (with a slight abuse of notation) 
the $L^1$-distance
\begin{equation}
\label{L1-dist}
\qquad\big\|\beta^{\bar\iota_1,\dots,\bar\iota_N}-\beta^{\iota_1,\dots,\iota_N}\big\|_{L^1(\R,\R^N)}
\doteq \sum_{k=1}^N \big\| \beta_k^{\bar\iota_k}-\beta_k^{\iota}\big\|_{L^1(\R,\R)}
\qquad\forall~(\bar\iota_1,\dots,\bar\iota_N), (\iota_1,\dots,\iota_N)\in (\lbrace{0,1\rbrace}^n)^N,
\end{equation}
on the set $\mathcal{B}_{n,h}$ in~\eqref{Bnhset-def}, 
and the usual $L^1$-distance on the set $\mathcal{F}_{n,h}$ in~\eqref{Fnhset-def}, there holds
\begin{equation}
\label{beta-phi-est}
\big\|\beta^{\bar\iota_1,\dots,\bar\iota_N}-\beta^{\iota_1,\dots,\iota_N}\big\|_{L^1(\R,\R^N)}
\leq 2 \big\|\phi^{\bar\iota_1,\dots,\bar\iota_N}-\phi^{\iota_1,\dots,\iota_N}\big\|_{L^1(\R,\Omega)}
\qquad\forall~(\bar\iota_1,\dots,\bar\iota_N), (\iota_1,\dots,\iota_N)\in (\lbrace{0,1\rbrace}^n)^N.
\end{equation}
Then, if we define $\mathcal{C}^{\mathcal{B}}_n(\varepsilon)$  as the maximum
number of elements in $\mathcal{B}_{n,h}$ that have $L^1$-distance (defined as in~\eqref{L1-dist}) less than~$\varepsilon$
from any given element of $\mathcal{B}_{n,h}$,
we deduce from~\eqref{beta-phi-est}
that $\mathcal{C}^{\mathcal{F}}_n(2\varepsilon)\leq \mathcal{C}^{\mathcal{B}}_n(4\varepsilon)$.
In turn, this inequality, together with~\eqref{CeF-est1}, yields
\begin{equation}
\label{CeF-est2}
N_{\varepsilon}(\mathcal{F}_{n,h}\ |\ L^{1}(\R,\Omega))\geq
\frac{2^{nN}}{\mathcal{C}^{\mathcal{B}}_n(4\varepsilon)}
\end{equation}
for all $h$ satisfying \eqref{assumption-h2}. \par
In order to provide an upper estimate on $\mathcal{C}^{\mathcal{B}}_n(4\varepsilon)$,
observe that, given any pair of $nN$-tuples $(\iota_1,\dots,\iota_N)$, $(\bar\iota_1,\dots,\bar\iota_N) \in (\lbrace{0,1\rbrace}^n)^N$,
letting $(\beta_1^{\iota_1},\dots,\beta_N^{\iota_N}), (\beta_1^{\bar\iota_1},\dots,\beta_N^{\bar\iota_N})$
denote the corresponding $N$-tuples in $\mathcal{B}_{n,h}$,
by definitions~\eqref{DefFiota}, \eqref{Bnhset-def}, \eqref{L1-dist},
and because every  interval $[\xi_i^-, \xi_i^+]$
has length $L$,
one has
\begin{equation}
\label{beta-dist-eq}
\big\|\beta^{\bar\iota_1,\dots,\bar\iota_N}-\beta^{\iota_1,\dots,\iota_N}\big\|_{L^1(\R,\R^N)}=
\frac{Lh}{n} \cdot d\big((\iota_1,\dots,\iota_N), (\bar\iota_1,\dots,\bar\iota_N)\big),
\end{equation}
where
\begin{equation*}
d\big((\iota_1,\dots,\iota_N), (\bar\iota_1,\dots,\bar\iota_N)\big) \doteq  
\mbox{Card} \Big\{ (k,\ell)\in\{1,\dots,N\}\times\{ 1,\dots,n \} \ \big|   \ (\iota_{k})_\ell \neq (\overline{\iota}_{k})_\ell \Big\}.
\end{equation*}
Then, given any fixed $nN$-tuple $\bar\iota\doteq(\bar\iota_1,\dots,\bar\iota_N)\in(\lbrace{0,1\rbrace}^n)^N$, 
define 
\begin{equation}
\label{CeI-def}
\mathcal{C}^{\mathcal{I}}_{n}(\varepsilon)
\doteq
\mbox{Card} \Big\{(\iota_1,\dots,\iota_N)\in(\{0,1\}^n)^N \  | \ 
d\big((\iota_1,\dots,\iota_N), (\bar\iota_1,\dots,\bar\iota_N)\big)\leq\varepsilon
\Big\}.
\end{equation}
Notice that the number $\mathcal{C}^{\mathcal{I}}_{n}(\varepsilon)$ is independent
of the choice of $\bar\iota\doteq(\bar\iota_1,\dots,\bar\iota_N)\in(\lbrace{0,1\rbrace}^n)^N$, 
and that, by~\eqref{beta-dist-eq}, there holds
\begin{equation}
\label{L1-dist-equiv}
\mathcal{C}^{\mathcal{B}}_n(4\varepsilon)=
\mathcal{C}^{\mathcal{I}}_{n}\bigg(\frac{4n\varepsilon}{Lh}\bigg).
\end{equation}
We next derive an upper bound on $\mathcal{C}^{\mathcal{I}}_{n}(\varepsilon)$ 
following the same strategy as in the proof of Proposition 2.2 in~\cite{AOK}.
Namely, by standard combinatorial properties, counting the $nN$-tuples that differ for a given number
of entries, we compute
\begin{equation}
\label{Ce}
\mathcal{C}^{\mathcal{I}}_{n}\bigg(\frac{4n\varepsilon}{Lh}\bigg) = \sum_{\ell=0}^{\left\lfloor \frac{4n \varepsilon}{Lh} \right\rfloor} \binom{nN}{\ell},
\end{equation}
where $\lfloor \alpha\rfloor \doteq \max\{z\in \Z\, \ | \  z\leq \alpha\}$ denotes the integer part of $\alpha$.
Next, observe that if $X_{1}, \dots, X_{nN}$ are independent random variables with Bernoulli distribution ${\mathbb P}(X_{i}=0)={\mathbb P}(X_{i}=1)= \frac{1}{2}$, then for any integer $k\leq nN$ one has
\begin{equation}
\label{Ce2}
{\mathbb P}( X_{1} + \dots + X_{nN} \leq k) = \frac{1}{2^{nN}} \sum_{\ell=0}^{k} \binom{nN}{\ell}.
\end{equation}
Now, we recall Hoeffding's inequality (\cite[Theorem 2]{Hoeffding})
which guarantees that, setting $S_{nN} \doteq X_{1} + \dots + X_{nN}$,
for any fixed $\mu>0$ there holds
\begin{equation}
\label{H}
{\mathbb P} (S_{nN} - {\mathbb{E}}(S_{nN}) \leq -\mu)  \leq \exp \left( - \frac{2 \mu^{2}}{nN}\right),
\end{equation}
where ${\mathbb{E}}(S_{nN})$ denotes the expectation of $S_{nN}$. 
Notice that by the above assumptions on $X_{1}, \dots, X_{nN}$,
we have  ${\mathbb{E}}(S_{nN})=\frac{nN}{2}$. Hence, taking 
$k\doteq\lfloor \frac{4n \varepsilon}{Lh} \rfloor$, $\mu\doteq \frac{nN}{2}- \lfloor  \frac{4n \varepsilon}{Lh} \rfloor$, 
and assuming
\begin{equation} 
\label{CondHoeffding}
\varepsilon
\leq \frac{NLh}{8},
\end{equation}
which implies $\mu>0$,
we deduce from~\eqref{L1-dist-equiv}-\eqref{H} that there holds
\begin{equation}
\label{EstCesp}
\begin{aligned}
\mathcal{C}^{\mathcal{B}}_n(4\varepsilon)
& \leq 2^{nN}\cdot
\exp \left(- 2 \frac{(\frac{nN}{2}-\lfloor \frac{4n \varepsilon}{Lh} \rfloor)^{2}}{nN}\right) 
\\
\noalign{\smallskip}
& \leq 2^{nN}\cdot \exp \left(- \frac{nN}{2} \left( 1- \frac{8\varepsilon}{LhN}\right)^{2}  \right) .
\end{aligned}
\end{equation}
In turn, \eqref{EstCesp} together with~\eqref{CeF-est2} yields 
\begin{equation}
\label{CeF-est3}
N_{\varepsilon}(\mathcal{F}_{n,h}\ |\ L^{1}(\R,\Omega))\geq
\exp \left(\frac{nN}{2} \left( 1- \frac{8\varepsilon}{LhN}\right)^{2}  \right),
\end{equation}
for all $n\geq 2$ and $h$ satisfying~\eqref{assumption-h2}, \eqref{CondHoeffding}.
In order to derive the largest lower bound on the right-hand side of~\eqref{CeF-est3} we maximize the map
$$\Psi(h,n)\doteq \frac{nN}{2}\left( 1- \frac{8\varepsilon}{LhN}\right)^2,$$
with  $h, n$,
subject to~\eqref{assumption-h}, \eqref{assumption-h2}, \eqref{CondHoeffding}.
If we first fix $n\geq 2$, and then optimize the map $h\mapsto \Psi(h,n)$,  when $h$ satisfies the bound~\eqref{assumption-h},  we find that the maximum is attained for
\begin{equation} 
\label{Defh}
h_n \doteq  \frac{Lb}{2n}.
\end{equation}
Next, optimizing the map $n\mapsto \Psi(h_n,n)$, with $h_n$ satisfying~\eqref{CondHoeffding}, 
i.e. with $n\leq \frac{NL^2 b}{16\varepsilon}$,
we deduce that the maximum is attained for 
\begin{equation} 
\label{Defn}
\overline n\doteq \bigg\lfloor \frac{N L^2 b}{48 \varepsilon}\bigg\rfloor +1.
\end{equation}
One can check that
\begin{equation*}
h_{\overline n} = \frac{Lb}{2 \overline n}\leq \frac{24 \varepsilon}{NL}\,,
\qquad\qquad
\frac{NLh_{\overline n}}{8} = \frac{NL^2 b}{16\overline n} \geq \frac{3\varepsilon}{2}\,,
\end{equation*}
so that, with $h_{\overline n},\, \overline n$
defined by~\eqref{Defh}, \eqref{Defn}, all conditions~\eqref{assumption-h}, \eqref{assumption-h2}, \eqref{CondHoeffding}
are verified, provided that $\varepsilon$
satisfies~\eqref{cond1}. Hence, we deduce from~\eqref{CeF-est3}
that
\begin{equation}
\label{NeF-est4}
N_{\varepsilon}(\mathcal{F}_{\overline n,h_{\overline n}}\ |\ L^{1}(\R,\Omega))\geq
\exp\big(\Psi(h_{\overline n}, \overline n)\big) =
\exp\bigg(\frac{L^2N^2 b}{216}\cdot\frac{1}{\varepsilon}\bigg),
\end{equation}
which, in turn, yields
\begin{equation}
\label{NeF-est5}
H_{\varepsilon}(\mathcal{F}_{\overline n,h_{\overline n}}\ |\ L^{1}(\R,\Omega))\geq
\frac{L^2N^2 b}{216\, \ln (2)}\cdot\frac{1}{\varepsilon}
\end{equation}
for all $\varepsilon$ satisfying~\eqref{cond1}.
By the above observations at Point~1., recalling~\eqref{ent-ineq}, this concludes the proof of the proposition.
\end{proof}
\begin{remark} 
\label{rem-scalar-combinatory}
When $N=1$, under the same assumptions as Proposition~\ref{controllability}
the following holds.
Given $L, M, b, T>0$, 
for every $0<\varepsilon \leq LM/12$, letting $\mathcal{A}_{[L,M,b,T]}$ be the set defined in~\eqref{attainable set-T}
one has
\begin{equation}
\label{low-bound-entr-scalar}
H_{\varepsilon}\Big(\mathcal{A}_{[L,M,b, T]}\ |\ L^1(\mathbb{R})\Big)\geq \frac{L^2 b}{108\, \ln (2)}\cdot\frac{1}{\varepsilon}.
\end{equation}
\end{remark}
\medskip

In order  to analyze the 
$\varepsilon$-entropy of solutions to Temple systems of conservation laws,
we shall now provide a lower bound on the $\varepsilon$-entropy 
of the class of maps $\mathcal{A}^w_{[L,M,b, T]}$ introduced in~\eqref{attainable set-T-w}.
Here, we are considering the 
$\varepsilon$-entropy of $\mathcal{A}^w_{[L,M,b, T]}$ related to the topology induced by the $L^1$-norm $\|w\|_{L^1}\doteq \sum_i \|w_i\|_{L^1}$.
\begin{proposition}
\label{Combinatory-temple}
Given $L, M, b, T>0$,
for every $\varepsilon >0$ satisfying 
\begin{equation}
\label{cond1-temple}
\varepsilon\leq\frac{LNM}{12},
\end{equation}
letting $\mathcal{A}^w_{[L,M,b, T]}$ be the set defined in~\eqref{attainable set-T-w},
one has
\begin{equation}
\label{low-bound-entr-A-temple}
H_{\varepsilon}\Big(\mathcal{A}^w_{[L,M,b, T]}\ |\ L^1(\mathbb{R},\Pi)\Big)\geq \frac{L^2N^2 b}{108\, \ln (2)}\cdot\frac{1}{\varepsilon}.
\end{equation}
\end{proposition}
\begin{proof}
The lower bound~\eqref{low-bound-entr-A-temple} is established with similar arguments to those of
the proof of Proposition~\ref{Combinatory}. Namely, given any integer $n\geq 2$ and any constant $h$
satisfying~\eqref{assumption-h},
we consider the set $\mathcal{B}_{n,h}$ introduced in~\eqref{Bnhset-def}.
Observe that by definitions~\eqref{attainable set-T-w}, \eqref{Bnhset-def} one has
\begin{equation}
\label{ent-ineq-temple}
N_{\varepsilon}\Big(\mathcal{A}_{[L,M,b, T]}\ |\ L^1(\mathbb{R},\Pi)\Big)\geq
N_{\varepsilon}\Big(\mathcal{B}_{n,h}\ |\ L^1(\mathbb{R},\Pi)\Big).
\end{equation}
Next, let $\mathcal{C}^{\mathcal{B}}_n(\varepsilon)$  denote the maximum
number of elements in $\mathcal{B}_{n,h}$ that have $L^1$-distance (defined as in~\eqref{L1-dist}) less than~$\varepsilon$
from any given element of $\mathcal{B}_{n,h}$. With the same combinatory arguments of the proof of Proposition~\ref{Combinatory},
for all
\begin{equation} 
\label{CondHoeffding-2}
\varepsilon
\leq \frac{NLh}{4},
\end{equation}
we derive
\begin{equation}
\label{CeB-est-temple}
N_{\varepsilon}\Big(\mathcal{B}_{n,h}\ |\ L^{1}(\R,\Pi)\Big)
\geq\frac{2^{nN}}{\mathcal{C}^{\mathcal{B}}_n(2\varepsilon)}
\geq\exp\big(\Psi(h,  n)\big),
\end{equation}
with 
\begin{equation}
\label{psi-def2}
\Psi(h,n)\doteq \frac{nN}{2}\left( 1- \frac{4\varepsilon}{LhN}\right)^2.
\end{equation}
Maximizing the map~\eqref{psi-def2} when $h, n$ are subject to~\eqref{assumption-h}, \eqref{CondHoeffding-2}, 
and combining \eqref{ent-ineq-temple}, \eqref{CeB-est-temple}, we find
\begin{equation}
\label{low-bound-entr-A-temple-2}
N_{\varepsilon}\Big(\mathcal{A}^w_{[L,M,b, T]}\ |\ L^1(\mathbb{R},\Pi)\Big)\geq \exp\big(\Psi(h_{\overline n}, \overline n)\big),
\end{equation}
with
\begin{equation} 
\label{Defn-2}
\overline n\doteq \bigg\lfloor\frac{N L^2 b}{24 \varepsilon}\bigg\rfloor+1,
\qquad\qquad
h_{\overline n} = \frac{Lb}{2 \overline n}\leq \frac{12 \varepsilon}{NL}.
\end{equation}
Finally, observing that
$$
\Psi(h_{\overline n}, \overline n) = \frac{L^2N^2 b}{108}\cdot\frac{1}{\varepsilon},
\qquad\qquad
\frac{NLh_{\overline n}}{4} \geq \frac{3\varepsilon}{2},
$$
and taking the logarithm of both sides of~\eqref{low-bound-entr-A-temple-2},
we recover the estimate~\eqref{low-bound-entr-A-temple}
for all $\varepsilon >0$ satisfying~\eqref{cond1-temple}.
\end{proof}

\subsection{Conclusion of the proofs of Theorem~\ref{mainthm}-(i) and Theorem~\ref{templethm}-(i)}

\bigskip

\noindent
{\it Proof of Theorem~\ref{mainthm}-(i).}
We shall provide a proof of the lower bound~\eqref{entrlowerbound}
for sets of functions of the form~\eqref{Cclass} with support contained in the interval $I_L\doteq [-L, L]$.
The case of sets of functions supported in any other given interval $I$ of length $|I|=2L$ can be recovered 
observing that every function in $S_T\big(\mathcal{L}_{[I,m,M]})$ is obtained by shifting 
horizontally a corresponding- function in $S_T\big(\mathcal{L}_{[I_L,m,M]})$ by a fixed constant.
Thus, the $\varepsilon$-entropy of the two
sets turns out to be the same.

Combining Proposition~\ref{controllability} and Proposition~\ref{Combinatory} we find
that, for every
\begin{equation}
\label{eps-bound}
0<\varepsilon\leq
\min\bigg\lbrace{\frac{L\, \overline d}{48 \alpha_4\, e^{\alpha_2/\alpha_3}}, \, 
\frac{L\,\Delta_{\wedge}\lambda}{96 \alpha_1\,\alpha_4 \, e^{\alpha_2/\alpha_3}}, \, 
\frac{L\, M}{48\alpha_4 \, e^{\alpha_2/\alpha_3}},\, \frac{N\, m}{48},\,
\frac{L\,N}{48 \alpha_6}
\bigg\rbrace}\cdot
\min \bigg\lbrace{\frac{1}{1+\alpha_5},\, \frac{T\!\cdot\! \Delta_{\wedge}\lambda}{L}\bigg\rbrace},
\end{equation}
($\alpha_l$, $l=2,\dots, 6$ and $\Delta_{\wedge}\lambda$ being the constants defined 
in~\eqref{alpha-1-2-def2}, \eqref{alfa4def}, \eqref{a5-def} and \eqref{T-assumption1}, respectively)
there holds
\begin{equation}
\label{H-est-final}
H_{\varepsilon}\Big(S_T\big(\mathcal{L}_{[I_L,m,M]}\cap{\mathcal{D}_0}\big)\ |\ L^1(\mathbb{R},\Omega)\Big)\geq 
\frac{{\widetilde L}^2N^2 b}{216\, \ln (2)}\cdot\frac{1}{\varepsilon},
\end{equation}
with
\begin{equation}
\label{def-Lt-b2}
\widetilde L = L\cdot \min\bigg\{\overline c_1, \overline c_2 \frac{T}{L}\bigg\},
\qquad\qquad 
b= \frac{1}{T}\cdot \frac{1}{\max\big\{\overline c_3,\, \overline c_4 \frac{N^2 L}{T},\, \overline c_5 \frac{N L}{\delta_0 T}\big\}},
\end{equation}
where
\begin{equation}
\label{c1234def}
\begin{gathered}
\overline c_1\doteq \frac{\overline c_2}{\overline c_2+\lambda_N(0)-\lambda_1(0)}
\qquad\quad\qquad \overline c_2\doteq
\displaystyle{\min_i}\big\{\lambda_{i+1}(0)-\lambda_i(0)\big\},
\\
\noalign{\bigskip}
\overline c_3\doteq 2\sup\big\{|\nabla\lambda_i(u)|\,; \ u\in B_{\overline d},\ i=1,\dots,N\,\big\}\,,
\qquad \overline c_5 \doteq 8 \sup \big\{ \big|l_i(0)\big| ; \ u\in B_{\overline d},\ i=1,\dots,N\, \big\}\,
\end{gathered}
\end{equation}
\begin{multline} \label{c1234defb}
\overline c_4\doteq \frac{  \,\overline c_5}{2 \, \overline c_2} \cdot
\Bigg(
\sup \bigg\{ \big|\lambda_k(u)-\lambda_j(u)\big| \big|l_i^T(u) Dr_k(u)\big|; \ u \in B_{\overline d},\ {i,j,k} \in \{ \, 1, \dots,N\,\big\} \bigg\} \\
+ \sup \big\{ \big|\nabla\lambda_i(u)\big| ; \ u\in B_{\overline d},\ i=1,\dots,N\, \big\}\, 
\Bigg).
\end{multline}
Thus, \eqref{H-est-final}-\eqref{c1234def}-\eqref{c1234defb} together yield~\eqref{entrlowerbound},
taking 
\begin{equation}
\label{c-def}
c_l = \overline c_l, \quad \text{for} \ l=1,2,\qquad\quad
c_l = 216\,\ln(2)\cdot\overline c_l, \quad  \text{for} \ l=3,4,5.
\end{equation}
\qed
\bigskip

\noindent
{\it Proof of Theorem~\ref{templethm}-(i).}
As for the proof of Theorem~\ref{mainthm}-(ii), it will be sufficient to establish the 
lower bound~\eqref{entrlowerbound-temple}
for sets of functions of the form~\eqref{Cclass-w} with support contained in the interval $I_L\doteq [-L, L]$.

Combining Proposition~\ref{controllability-temple} and Proposition~\ref{Combinatory-temple} we find
that, for every
\begin{equation}
\label{eps-bound2}
0<\varepsilon\leq
\min\bigg\lbrace{\frac{LN\,\overline d'}{12},\, \frac{L\sqrt N \,\Delta_{\wedge}\lambda}{24 \,\alpha_1''}, \, 
\frac{LN M}{12},\, \frac{N m}{24}
\bigg\rbrace},
\end{equation}
($\alpha''_1, \Delta_{\wedge}\lambda$ being the constants defined 
in~\eqref{alpha-2w-def1}, \eqref{T-assumption1}, respectively)
there holds
\begin{equation}
\label{H-est-final2}
H_{\varepsilon}\Big(S^w_T\big(\mathcal{L}^w_{[I_L,m,M]}\big)\ |\ L^1(\mathbb{R},\Pi)\Big)\geq 
\frac{{L}^2N^2 b}{108\, \ln (2)}\cdot\frac{1}{\varepsilon},
\end{equation}
with
\begin{equation}
\label{def-Lt-b3}
b= \frac{1}{T}\cdot \frac{1}{\max\big\{\overline c_6,\, \overline c_7 \frac{N L}{T}\big\}},
\end{equation}
where
\begin{equation}
\label{c1234def2}
\begin{gathered}
\overline c_6\doteq 2\sup\bigg\{\big|\nabla\lambda_i(u)\cdot r_i(u)\big|\ ; \ u\in B_{\overline d}, \ i=1,\dots,N\bigg\},
\\
\noalign{\medskip}
\overline c_7\doteq \frac{2}{\overline c_2}\cdot
\sup\bigg\{
\big|\nabla\lambda_i(u)\cdot r_j(u)\big|\ ; \ u\in B_{\overline d}, \ i,j=1,\dots,N\bigg\}\,.
\end{gathered}
\end{equation}
Thus, \eqref{H-est-final2}-\eqref{c1234def2} together yield~\eqref{entrlowerbound-temple},
taking 
\begin{equation}
\label{c-def2}
c_l = 108\,\ln(2)\cdot\overline c_l, \quad  \text{for} \ l=6,7.
\end{equation}

\qed
\bigskip

\section{Upper compactness estimates for genuinely nonlinear Temple systems}
\label{sec:upper-est-temple}
%
Assume that~\eqref{conlaws} is a strictly hyperbolic system of Temple class, 
and that all  characteristic families are genuinely nonlinear (cf. subsection~\ref{subsec-control}). Let $(S_t^w)_{t\geq 0}$
be the Riemann coordinate expression of the semigroup of entropy weak solutions generated by~\eqref{conlaws},
defined on a domain $L^1(\R,\Pi)$ with \linebreak $\Pi \doteq [a_1, b_1]\times\cdots\times [a_N,b_N]$.
In connection with a class of initial data $\mathcal{L}^w_{[I,m,M]}\subset L^1(\R,\Pi)$
as in~\eqref{Cclass-w},
consider the sets of $i$-th components of elements of $S^w_T\big(\mathcal{L}^w_{[I,m,M]}\big)$,
at a fixed time $T>0$:
\begin{equation}
\label{STi}
S^w_{T,i}\big(\mathcal{L}^w_{[I,m,M]}\big)
\doteq\big\lbrace{\varphi_i\ |\ (\varphi_1,\dots,\varphi_N)\in S^w_T(\mathcal{L}^w_{[I,m,M]})\big\rbrace},
\qquad i=1,\dots,N.
\end{equation}
Thanks to the Ole\v{\i}nik-type inequalities~\eqref{Olest},
we may establish an upper estimate on the $\varepsilon$-entropy 
for $S^w_{T,i}\big(\mathcal{L}^w_{[I,m,M]}\big)$
following 
the same strategy adopted in~\cite{DLG} for scalar conservation laws
with convex flux, 
relying  on the upper bound on the $\varepsilon$-entropy for classes of nondecreasing functions provided by:

\begin{lemma}$\mathrm{(\cite[Lemma~3.1]{DLG})}$
\label{LemDLG}
Given any, $L, M>0$, setting
\begin{equation}
\label{I-class-def}
{\mathcal I}_{[L,M]} \doteq \{ v : [0,L] \rightarrow [0,M] \ | \ v \text{ is nondecreasing}\ \},
\end{equation}
for $0< \varepsilon < \frac{LM}{6}$, there holds
\begin{equation}
\nonumber
H_{\varepsilon} ( {\mathcal I}_{[L,M]} \ | \ L^{1}([0,L])) \leq \frac{4 LM}{\varepsilon} .
\end{equation}
\end{lemma}
\medskip
\noindent
In order to obtain an a-priori bound on size of the support of solutions to~\eqref{conlaws},
expressed in terms of the $L^1$-norm of their initial data, 
we will use the next technical
lemma derived in~\cite{AOK}.
\begin{lemma}$\mathrm{(\cite[Lemma~4.2]{AOK})}$
\label{LemEstSLLinf2}
Given $v \in \mathrm{BV}(\R)$, compactly supported and satisfying 
\begin{equation} \label{EqvCondSL}
D\, v \leq B \ \ \text{ in the sense of measures,}
\end{equation}
for some constant $B>0$, there holds
\begin{equation} \label{Eqestv2}
\| v \|_{\!\strut L^{\infty}} \leq \sqrt{2B \| v \|_{L^1}}.
\end{equation}
\end{lemma}
\bigskip

\noindent
{\it Proof of Theorem~\ref{templethm}-(ii).}
As for the proof of Theorem~\ref{templethm}-(i), it will be sufficient to establish the 
upper bound~\eqref{entrupperbound-temple}
for sets of functions of the form~\eqref{Cclass-w} with support contained in the interval $I_L\doteq [-L, L]$.
As stated in the introduction, we adopt the norms $\|w\|_{L^1}\doteq \sum_i \|w_i\|_{L^1}$,
$\|w\|_{L^\infty}\doteq \sup_i \|w_i\|_{L^\infty}$
on the space $L^1(\R, \Pi)$.
\medskip

\noindent
{\bf 1.} Given any initial data $\overline w\in \mathcal{L}^w_{[I_L,m,M]}$,
let $w(t,x)\doteq S^w_t \overline w (x)$ be the corresponding entropy
weak solution of~\eqref{conlaws} satisfying the Ole\v{\i}nik-type inequalities~\eqref{Olest}.
Observe that, by the properties of solutions of Temple systems
(cf.~\cite{brgoa1}), and because $\overline w\in \mathcal{L}^w_{[I_L,m,M]}$, for all $t\geq 0$, $i=1,\dots, N$, one has
\begin{equation}
\label{est-w-linf-l1-temple}
\|w_i(t,\cdot)\|_{L^\infty(\R,\Pi)}\leq \|\overline w_i\|_{L^\infty(\R,\Pi)}\leq M,
\qquad\quad
\|w_i(t,\cdot)\|_{L^1(\R,\Pi)}\leq \|\overline w_i\|_{L^1(\R,\Pi)}\,.
\end{equation}
On the other hand, notice that $w_i(t,\cdot)$ is compactly supported,
and that by virtue of~\eqref{Olest}, \eqref{est-w-linf-l1-temple},
one has $w_i(t,\cdot)\in \text{BV}(\R)$ for all $t>0$
and
\begin{equation}
\label{Ol-est2}
Dw_i(t,\cdot)\leq\frac{1}{c\,t}
\qquad\quad\forall~t>0, \quad i=1,\dots,N\,.
\end{equation}
Thus, invoking Lemma~\ref{LemEstSLLinf2}
and relying on~\eqref{est-w-linf-l1-temple}, \eqref{Ol-est2},
we derive
\begin{equation}
\label{est-w-l1-temple2}
\sum_i \|w_i(t,\cdot)\|_{L^\infty}
\leq \sum_i \sqrt{\frac{2 \|\overline w_i\|_{L^1}}{c\,t}}
%
\leq \sqrt{\frac{2N}{c\,t}}\, \sqrt{\sum_i  \|\overline w_i\|_{L^1}}
\leq \sqrt{\frac{2 N m}{c\,t}}\,,
\qquad\quad\forall~t>0\,.
\end{equation}
Moreover, applying the theory of generalized characteristics
 (see~\cite[Section 10.2]{Dafermos:Book}), letting $\xi_{(t,z)}^-(\cdot), \xi_{(t,z)}^+(\cdot)$ 
 denote the minimal and maximal backward characteristics emanating from $(t,z)$, 
 and setting
 \begin{equation}
 \label{lminplus}
 l^-(t)\doteq\inf\big\{z\,\big|\ \xi_{(t,z)}^+(0)\geq -L\big\}\,,
 \qquad\qquad
 l^+(t)\doteq\sup\big\{z\,\big|\ \xi_{(t,z)}^-(0)\leq L\big\}\,,
 \end{equation}
 we find
\begin{equation}
\label{supp-wi-temple}
\mathrm{Supp}(w_i(t,\cdot))\subset [l_i^-(t),\, l_i^+(t)]\,,
\end{equation}
for all $t\geq 0$, $i=1,\dots, N$. Then, recalling that the minimal  backward characteristic
$\xi_{(t,z)}^-(\cdot)$ is a solution of 
\begin{equation}
\dot \xi (s) = \lambda_i\big(w(s,\xi (s)-)\big)\quad \text{a.e.}\quad s\in[0,t]\,,
\end{equation}
setting 
\begin{equation}
\label{aij-def}
\alpha_{i,j}\doteq \sup\big\{
|\nabla \lambda_i(u)\cdot r_j(u)|\ ;\,  |W(u)|\leq M \big\},
\end{equation}
and relying on~\eqref{est-w-l1-temple2}, we derive 
\begin{equation}
\label{supp-wi-temple2}
\begin{aligned}
l_i^+(t) &\leq L+\lambda_i(0) \cdot t + \sum_j \alpha_{i,j}\int_0^t \|w_j(s,\cdot)\|_{L^\infty}~ds
\\
\noalign{\smallskip}
&\leq L+\lambda_i(0) \cdot t + \sup_j \alpha_{i,j}\int_0^t \sum_j  \|w_j(s,\cdot)\|_{L^\infty}~ds
\\
\noalign{\smallskip}
&\leq L+\lambda_i(0) \cdot t + \sup_j \alpha_{i,j}\sqrt{\frac{2 N m}{c}}\int_0^t \frac{1}{\sqrt s}~ds
\\
\noalign{\smallskip}
&\leq L+\lambda_i(0) \cdot t + \sup_j \alpha_{i,j}\sqrt{\frac{8N m \,t}{c}}\,,
\end{aligned}
\end{equation}
for all $t\geq 0$, $i=1,\dots, N$. 
Analogously, observing that the maximal backward characteristic
$\xi_{(t,z)}^+(\cdot)$ is a solution of
\begin{equation}
\qquad\qquad \dot \xi (s) = \lambda_i\big(w(s,\xi (s)+)\big)\qquad \text{a.e.}\quad s\in[0,t]\,,
\end{equation}
with the same arguments above we derive
\begin{equation}
\label{supp-wi-temple3}
l_i^-(t) \geq -L+\lambda_i(0) \cdot t - \sup_j \alpha_{i,j}\sqrt{\frac{8N m \,t}{c}}
\qquad\forall~t\geq 0, \quad i=1,\dots,N\,,
\end{equation}
which, together with~\eqref{supp-wi-temple}, \eqref{supp-wi-temple2}, yields
%
\begin{equation}
\label{supp-wi-temple4}
\mathrm{Supp}(w_i(t,\cdot))\subset [-L_t+\lambda_i(0) \cdot t,\ L_t+\lambda_i(0) \cdot t]\,,
\qquad L_t\doteq \bigg(L+\sup_{i,j} \alpha_{i,j}\sqrt{\frac{8N m \,t}{c}}\ \bigg)\,,
\end{equation}
for all $t\geq 0$, $i=1,\dots, N$. 
Finally, observing that by~\eqref{c-def-temple}, \eqref{aij-def} we have $c\leq \sup_{i,j} \alpha_{i,j}$,
and combining \eqref{est-w-l1-temple2}  with~\eqref{supp-wi-temple4}, we find
\begin{equation}
\label{est-w-linf-temple3}
\|w_i(t,\cdot)\|_{L^\infty}\leq \frac{L_t}{c\, t}
\qquad\forall~t> 0, \quad i=1,\dots,N\,.
\end{equation}
\medskip

\noindent
{\bf 2.} In connection with any given 
$\psi\in S^w_T\big(\mathcal{L}^w_{[I_L,m,M]}\big)$, consider the function $\varphi_i^\natural : [0, \, 2L_T]\to \R$
defined by setting
\begin{equation}
\varphi_i^\natural(x)\doteq \frac{x}{c\,T}-\psi\big(x+\lambda_i(0) \cdot T- L_T\big)+\frac{L_T}{c\, T},
\end{equation}
with $L_T$ as in~\eqref{supp-wi-temple3}. Notice that, by virtue of~\eqref{Ol-est2}, 
$\varphi_i^\natural$ is nondecreasing and, thanks to~\eqref{est-w-linf-temple3}, one has
\begin{equation}
\label{psi-nat-est}
0\leq \varphi_i^\natural(x)\leq \frac{4 L_T}{c\,T}\qquad\quad\forall~x\in  [0, \, 2L_T]\,.
\end{equation}
Hence, recalling the definition~\eqref{I-class-def}, we have
\begin{equation*}
\varphi_i^\natural\in \mathcal{I}_{[2L_T,\, \frac{4 L_T}{c\,T}]}\,.
\end{equation*}
Finally, observe that since $\varphi_i^\natural$ is obtained from $\varphi_i$
by a change of sign, a translation by a fixed function, and a shift of a
fixed constant, it follows that, setting 
\begin{equation*}
\mathcal{U}_i^\natural\doteq \big\{\varphi_i^\natural\ \big| \ \varphi\in S^w_T\big(\mathcal{L}^w_{[I_L,m,M]}\big)\big\}\,,
\end{equation*}
recalling~\eqref{STi}, there holds
\begin{equation}
\label{entr-eq-restr}
\begin{aligned}
N_{\varepsilon}\Big(S^w_{T,i}\big(\mathcal{L}^w_{[I,m,M]}\big)\ |\ L^1([-L_T+\lambda_i(0) \cdot T,\ L_T+\lambda_i(0) \cdot T])\Big)
&=N_{\varepsilon}\Big(\mathcal{U}_i^\natural \ |\ L^1([0, \, 2L_T])\Big)
\\
\noalign{\smallskip}
&\leq  N_{\varepsilon}\Big(\mathcal{I}_{[2L_T,\, \frac{4 L_T}{c\,T}]} \ |\ L^1([0, \, 2L_T])\Big)\,.
\end{aligned}
\end{equation}
On the other hand, by virtue of~\eqref{supp-wi-temple3}, one has
\begin{equation}
\label{entr-eq-nat}
N_{\varepsilon}\Big(S^w_{T,i}\big(\mathcal{L}^w_{[I,m,M]}\big)\ |\ L^1(\mathbb{R})\Big)=
N_{\varepsilon}\Big(S^w_{T,i}\big(\mathcal{L}^w_{[I,m,M]}\big)\ |\ L^1([-L_T+\lambda_i(0) \cdot T,\ L_T+\lambda_i(0) \cdot T])\Big).
\end{equation}
Thus, applying Lemma~\ref{LemDLG}, and relying on~\eqref{entr-eq-restr}, \eqref{entr-eq-nat}, we find
\begin{equation}
N_{\varepsilon}\Big(S^w_{T,i}\big(\mathcal{L}^w_{[I,m,M]}\big)\ |\ L^1(\mathbb{R})\Big)\leq 
2^{\frac{32 L_T^2}{c\,T}\cdot \frac{1}{\varepsilon}}
\qquad\quad\forall~i=1,\dots,N,
\end{equation}
which, in turn, yield
\begin{equation}
N_{\varepsilon}\Big(S^w_T\big(\mathcal{L}^w_{[I,m,M]}\big)\ |\ L^1(\mathbb{R},\Pi)\Big)\leq
\prod_{i=1}^N N_{\!\frac{\varepsilon}{N}} N_{\varepsilon}\Big(S^w_{T,i}\big(\mathcal{L}^w_{[I,m,M]}\big)\ |\ L^1(\mathbb{R})\Big)
\leq 2^{\frac{32 N^2 L_T^2}{c\,T}\cdot \frac{1}{\varepsilon}},
\end{equation}
proving the upper bound~\eqref{entrupperbound-temple}.
\qed
\bigskip
\section*{Acknowledgements}
FA and KTN are partially supported by the European Union
Seventh Framework Programme [FP7-PEOPLE-2010-ITN] under grant agreement n.264735-SADCO,
by the ERC Starting Grant 2009 n.240385 ConLaws, and by 
Fondazione CaRiPaRo Project "Nonlinear Partial Differential Equations: models, analysis, and control-theoretic problems".
\bigskip

\end{document}